\documentclass{amsart}

\usepackage[utf8]{inputenc}
\usepackage[english]{babel}
\usepackage{lmodern}
\usepackage[a4paper, margin=3cm]{geometry}
\usepackage{mathtools}
\usepackage{amsfonts}
\usepackage{amsthm}
\usepackage{amssymb}
\usepackage{listings}
\usepackage{graphicx}
\usepackage{tikz}
\usetikzlibrary{cd}
\usetikzlibrary{babel}
\usepackage{bbold}
\usepackage[colorinlistoftodos,textwidth=2cm]{todonotes}
\usepackage[font=small,labelfont=bf]{caption}
\usepackage[subrefformat=parens]{subcaption}
\usepackage{overpic}
\usepackage{enumerate}
\usepackage[T1]{fontenc}

\usepackage{latexsym}
\usepackage{pgf}
\usepackage{calrsfs,bbm}

\usepackage{amsmath}
\usepackage{theoremref}
\usepackage{mathrsfs}
\usepackage{amsthm}
\usepackage{mathtools}
\usepackage{enumitem}
\usepackage{tikz}
\usetikzlibrary{arrows,decorations.markings,knots,
	hobby,
	decorations.pathreplacing,
	shapes.geometric,
	calc}
\usepackage{float}
\usepackage{pgfkeys}

\usepackage{blkarray}
\usepackage{overpic}
\usepackage{soul}

\usepackage{enumitem}
\setlist{leftmargin=6.5mm}
\setenumerate[1]{label={(\roman*)}}
\setlist[enumerate]{leftmargin=8mm}

\graphicspath{ {images/} }
\usepackage{xcolor}
\usepackage[colorlinks,
    linkcolor={red!50!black},
    citecolor={blue!50!black},
    urlcolor={blue!80!black}]{hyperref}

\newcommand{\Z}{\mathbb{Z}}

\newcommand{\R}{\mathbb{R}}
\newcommand{\K}{\mathcal{K}}
\newcommand{\C}{\mathbb{C}}
\newcommand{\T}{\mathbb{T}}
\newcommand{\A}{\mathcal{A}}
\renewcommand{\Re}{\operatorname{Re}}

\DeclareMathOperator{\lk}{lk}
\newcommand{\sgn}{\operatorname{sgn}}
\newcommand{\sign}{\operatorname{sign}}
\newcommand{\nul}{\operatorname{null}}

\newcommand{\Ker}{\operatorname{Ker}}
\newcommand{\om}{\omega}
\newcommand{\bom}{\overline{\omega}}

\newcommand{\de}{\stackrel{\cdot}{=}}

\newcommand{\LambdaC}{\Lambda^{\mathbb{C}}}

\newtheorem{thm}{Theorem}[section]
\newtheorem{lem}[thm]{Lemma}
\newtheorem{prop}[thm]{Proposition}
\newtheorem{cor}[thm]{Corollary}

\theoremstyle{definition}
\newtheorem{definition}[thm]{Definition}

\newtheorem{exs}[thm]{Examples}

\theoremstyle{remark}
\newtheorem{rem}[thm]{Remark}
\newtheorem{rems}[thm]{Remarks}

\makeatletter
\@namedef{subjclassname@1991}{2020 Mathematics Subject Classification}
\makeatother

 \title{Extended signatures and link concordance}

\author{David Cimasoni}
\address{David Cimasoni -- Section de math\'ematiques, Universit\'e de Gen\`eve, Suisse}
\email{david.cimasoni@unige.ch}

\author{Livio Ferretti}
\address{Livio Ferretti -- Section de math\'ematiques, Universit\'e de Gen\`eve, Suisse}
\email{livio.ferretti@unige.ch}

\author{Iuliia Popova}
\address{Iuliia Popova -- Section de math\'ematiques, Universit\'e de Gen\`eve, Suisse}
\email{Iuliia.Popova@etu.unige.ch}

\begin{document}

\makeatletter
   \providecommand\@dotsep{5}
 \makeatother

\begin{abstract}
The Levine-Tristram signature admits a~$\mu$-variable extension for~$\mu$-component links: it was first defined as an integer valued function on~$(S^1\setminus\{1\})^\mu$, and recently extended to
the full torus~$\T^\mu$. The aim of the present article is to study and use this extended signature.
First, we show that it is constant on the connected components of the complement of the zero-locus of some renormalized Alexander polynomial. Then, we prove that the extended signature is a concordance invariant on an explicit dense subset
of~$\T^\mu$. Finally, as an application, we present an infinite family of~3-component links with the following property:
these links are not concordant to their mirror image, a fact that can be detected neither by the non-exended signatures, nor by the multivariable Alexander polynomial, nor by the Milnor triple linking number.
\end{abstract}



\maketitle

	\section{Introduction}
	\label{sec:intro}
	
The {\em Levine-Tristram signature\/}~\cite{Tro,Mur,Tri,Lev} of an oriented link~$L$ in the three-sphere~$S^3$ is the function
\[
\sigma_L\colon S^1\setminus\{1\}\longrightarrow \Z
\]
defined by~$\sigma_L(\omega)=\sigma(H(\omega))$,
the signature of the Hermitian matrix
\begin{equation}
\label{eq:H}
H(\omega)=(1-\omega)A+(1-\overline{\omega})A^T
\end{equation}
with~$A$ a Seifert matrix for~$L$. Similarly, the {\em Levine-Tristram nullity\/}
is the function~$\eta_L\colon S^1\setminus\{1\}\to \Z$ defined by~$\eta_L(\omega)=\eta(H(\omega))$, where~$\eta$ stands for the nullity.

These invariants enjoy numerous remarkable properties. For example, if~$-L$ denotes the mirror image of~$L$, then~$\sigma_{-L}=-\sigma_L$, so a non-zero signature ensures that the link is not amphicheiral. Also, the function~$\sigma_L$
is constant on the connected components of the complement of the roots of the Alexander polynomial~$\Delta_L$ in~$S^1\setminus\{1\}$. Furthermore, it provides lower bounds on the unknotting number of~$L$, on its splitting number, and on the minimal genus of an orientable surface~$S\subset S^3$ with oriented boundary~$\partial S=L$.
Finally, if~$\omega\in S^1\setminus\{1\}$ is not the root of any polynomial~$p\in\Z[t^{\pm 1}]$ with~$p(1)=\pm 1$, then~$\sigma_L(\omega)$
also yields a lower bound on the {\em topological four-genus\/} of~$L$, i.e. the minimal genus of a locally flat orientable surface~$F$ properly embedded in~$B^4$ with oriented boundary~$\partial F=L$. On this (dense) subset of~$S^1\setminus\{1\}$, the signature and nullity are actually known to be invariant under {\em topological concordance\/}~\cite{NP17}. We refer to the survey~\cite{Conway-survey} for more detailed information on these classical invariants, including references for the facts stated above.

\medskip

Similarly to the Alexander polynomial, the Levine-Tristram signature and nullity admit multivariable extensions. This story is best told in the setting of colored links, that we now recall.
Given a positive integer~$\mu$, a~$\mu$-{\em colored link} is an oriented link~$L$ in~$S^3$ each of whose components
is endowed with a {\em color\/} in~$\{1,\dots,\mu\}$ in such a way that all colors are used. We denote such a colored link by~$L=L_1\cup\dots\cup L_\mu$, where~$L_i$ is the union of the components of color~$i$. Two colored links are {\em isotopic\/} is they are related by an ambient isotopy which is consistent with the orientation and color of each component.
For example, a~$1$-colored link is just an oriented link, while a~$\mu$-component~$\mu$-colored link is an oriented ordered link.

The {\em multivariable signature\/} of a~$\mu$-colored link~$L$ is a function
\[
\sigma_L\colon (S^1\setminus\{1\})^\mu\longrightarrow \Z\,,\quad \omega=(\omega_1,\dots,\omega_\mu)\longmapsto\sigma(H(\omega))\,,
\]
where~$H(\omega)$ is a Hermitian matrix defined via {\em generalized Seifert matrices\/} associated with a generalized
Seifert surface for~$L$ called a {\em C-complex\/}~\cite{Cooper,CF08,DMO}. Similarly, the {\em multivariable nullity\/} of~$L$ is the map~$\eta_L\colon (S^1\setminus\{1\})^\mu\to\Z$ defined by~$\eta_L(\omega)=\eta(H(\omega))$.
In the~$\mu=1$ case, a C-complex is nothing but a Seifert surface for the oriented link~$L$, and~$H(\omega)$ is given
by~\eqref{eq:H}, justifying the notation and the terminology.

Remarkably, all of the properties of the Levine-Tristram signature mentioned above extend to the multivariable setting.
For example, the signature~$\sigma_L$ satisfies~$\sigma_{-L}=-\sigma_L$, and it is constant on the connected components of the complement in~$(S^1\setminus\{1\})^\mu$ of the zeros of the multivariable Alexander polynomial~$\Delta_L(t_1,\dots,t_\mu)$~\cite{CF08}.
Also, if there is no~$p\in\Z[t_1^{\pm 1},\dots,t_\mu^{\pm 1}]$ with~$p(1,\dots,1)=\pm 1$ and~$p(\omega_1,\dots,\omega_\mu)=0$, then~$\sigma_L(\omega_1,\dots,\omega_\mu)$ and~$\eta_L(\omega_1,\dots,\omega_\mu)$
are invariant under {\em topological concordance of colored links\/}~\cite{CNT}, see Definition~\ref{def:concordance}.

\medskip

The slightly peculiar domain of these functions leads to the following natural question: does there exist a sensible extension
of~$\sigma_L$ and~$\eta_L$ from~$(S^1\setminus\{1\})^\mu$ to the full torus~$(S^1)^\mu=:\T^\mu$~?
There is no obvious answer, as the standard definition via (generalized) Seifert matrices yields a zero signature and ill-defined nullity as soon as one coordinate of~$\omega$ is equal to~1. Similarly, the alternative definition pioneered in~\cite{Viro} using the
twisted homology of the exterior of a bounding surface in~$B^4$ is in general not well-defined on the full torus, see e.g.~\cite[Section~4,4]{DFL}.

A positive answer was given in the recent paper~\cite{CMP23}, via a rather technical construction. In a nutshell, the exterior~$X_L:=S^3\setminus\nu(L)$ of the~$\mu$-colored link~$L$ can be glued
along its boundary to an appropriate plumbed~$3$-manifold, yielding a closed oriented~$3$-manifold~$M_L$ that admits a so-called {\em meridional\/} homomorphism~$\varphi\colon H_1(M_L)\to\Z^\mu$. Slightly altering~$L$ if needed, there exists a~$4$-manifold~$W_F$ with~$\partial W_F=M_L$,
endowed with a homomorphism~$\Phi\colon H_1(W_F)\to\Z^\mu$ extending~$\varphi$. The extended signature and nullity
are then defined as the signature and nullity of~$W_F$ with coefficients twisted via~$\Phi$ and~$\omega$ (see Section~\ref{sub:ext} for details).
Note that the aim of~\cite{CMP23} was not these extensions {\em per se\/}, but their
use to understand the behavior of the non-extended signature and nullity when some~$\omega_i$ is close to~$1$. For example, in the~$\mu=1$ case, it was observed that~$\lim_{\omega\to 1}\sigma_L(\omega)$
is equal to the signature of the linking matrix of~$L$ as long as~$(t-1)^{\vert L\vert}$ does not divide~$\Delta_L$, where~$\vert L\vert$ stands for the number of components of~$L$. (This result was first obtained in~\cite{BZ22} via completely different methods, see also Corollary~\ref{cor:cont1} below.)
We refer the reader to~\cite[Theorem~1.6]{CMP23} for much broader results obtained in this fashion.

\medskip

The aim of the present work is to study these extended signature and nullity, and to apply them to link concordance.
We have three main results and one application, that we now summarize.

As a first result, we compute the greatest common divisor~$\widetilde\Delta_L\in\Z[t_1^{\pm 1},\dots,t_\mu^{\pm 1}]=:\Lambda$
of the first elementary ideal of the module~$H_1(M_L;\Lambda)$. This module should be thought of as a natural
renormalised version of the Alexander module~$H_1(X_L;\Lambda)$, and coincides with it in the case of knots.
Similarly, the polynomial~$\widetilde\Delta_L$ is a renormalized multivariable Alexander polynomial, which in the~$\mu=1$ case coincides with the {\em Hosokawa polynomial\/}~\cite{Hos}.
Therefore, the manifold~$M_L$ yields a geometric construction of the Hosokawa polynomial for~$\mu=1$, and allows to define an explicit multivariable version (see Proposition~\ref{prop:Hosokawa}, and Theorem~\ref{thm:intro-const} below).

Our second result is the following: the module~$H_1(M_L;\Lambda)$ yields a filtration
\[
\T^\mu\setminus\{(1,\dots,1)\}=\Sigma_0\supset\Sigma_1\supset\dots\supset\Sigma_{\ell-1}\supset\Sigma_\ell=\emptyset
\]
of the pointed torus by algebraic subvarieties so that for all~$r\ge 0$,~$\sigma_L$  is constant on the connected components of~$\Sigma_r\setminus\Sigma_{r+1}$ (see Theorem~\ref{thm:main}, which also deals with the behavior of~$\eta_L$). Together with the explicit determination of~$\widetilde{\Delta}_L$
from Proposition~\ref{prop:Hosokawa}, this implies the following result (Corollaries~\ref{cor:cont1} and~\ref{cor:cont2}).

\begin{thm}
\label{thm:intro-const}
	The extended Levine-Tristram signature~$\sigma_L\colon S^1\to\Z$ of an oriented link~$L$ is constant on
	the connected components of the complement of the zeros of the Hosokawa polynomial
	\[
	\widetilde\Delta_L(t)=\frac{\Delta_L(t)}{(t-1)^{|L|-1}}\,.
	\]
	The extended signature~$\sigma_L\colon \T^\mu\setminus\{(1,\dots,1)\}\to\Z$ of a~$\mu$-colored link~$L$ with~$\mu>1$ 
	is constant on the connected components of the complement of the zeros of the multivariable Hosokawa polynomial
	\[
	\widetilde\Delta_L(t_1,\dots,t_\mu)\stackrel{\cdot}{=}\prod_{i=1}^\mu (t_i-1)^{\nu_i}\Delta_L(t_1,\dots,t_\mu)\,,
	\]
	where $\nu_i=\Big(\sum_{\substack{K\subset L_i\\ K'\subset L\setminus L_i}}\vert\lk(K,K')\vert\Big)-\vert L_i\vert$.
\end{thm}

Our third result will not come as a surprise (see Definition~\ref{def:concordance} and Theorem~\ref{thm:conc}).

\begin{thm}
\label{thm:intro-inv}
If~$L$ and~$L'$ are topologically concordant~$\mu$-colored links, then~$\sigma_L(\omega)=\sigma_{L'}(\omega)$ and~$\eta_L(\omega)=\eta_{L'}(\omega)$ for all~$\omega\in\T^\mu$ not the root of any~$p\in\Z[t_1^{\pm 1},\dots,t_\mu^{\pm 1}]$ with~$p(1,\dots,1)=\pm 1$.
\end{thm}

Finally, to assess the power of these extended invariants, this result is applied to an explicit infinite family of examples. For any~$n\in\mathbb{N}$, consider the~$3$-colored link~$L(n)$ illustrated in Figure~\ref{fig:brunnian_links}. Using extended signatures, we show that for all~$n\neq 0$, the link~$L(n)$ is not concordant to its mirror image. Moreover,
this fact cannot be proved using the following standard obstructions: non-extended signatures, multivariable Alexander polynomials~\cite{Kaw78}, Blanchfield forms over the localised ring~$\Lambda_S$~\cite{Hil81}, linking numbers and Milnor triple linking numbers~\cite{Mil54,Sta65,Cas75}. However, it should be noted that this fact can also be detected by the Milnor number~$\mu(1123)$.

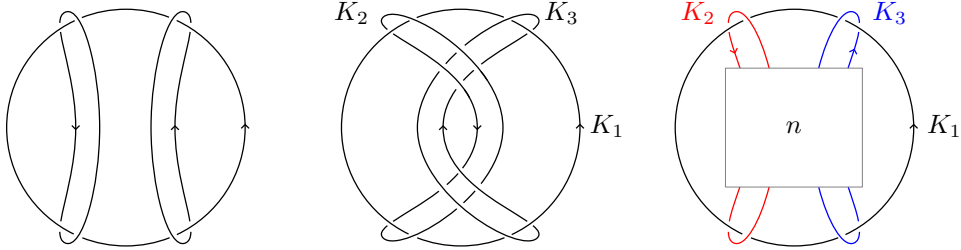
\begin{figure}
	\centering
	\begin{overpic}[width=.8\textwidth]{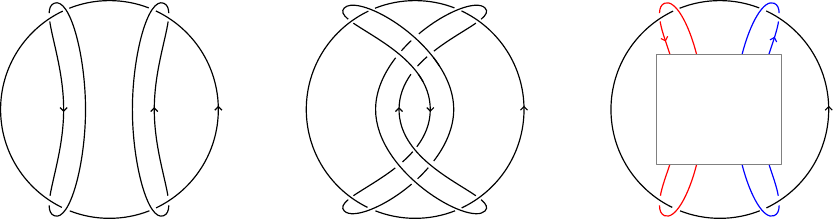}
		\put (85.5,12.5) {$n$}
		\put (64,12.5) {$K_1$}
		\put (59,25) {$K_3$}
		\put (36,25) {$K_2$}
		\put (101,12.5) {$K_1$}
		\put (95,25) {\color{blue}$K_3$}
		\put (74,25) {\color{red}$K_2$}
	\end{overpic}
	\caption{The links $L(n)$ for $n=0$ (on the left), and $n=1$ (in the middle). In the general case (on the right), the two bands corresponding to $K_2$ and $K_3$ twist $n$ times inside the grey box.}
	\label{fig:brunnian_links}
\end{figure}

\bigskip

This article is organised as follows. Section~\ref{sec:background} recalls the (rather substantial) background material.
Section~\ref{sec:Hosokawa} deals with the determination of the multivariable Hosokawa polynomial associated to~$M_L$. Section~\ref{sec:cont} contains the statement of the piecewise-continuity of~$\sigma_L$ and~$\eta_L$
along strata of~$\T^\mu$, together with the corollaries stated as Theorem~\ref{thm:intro-const} above. In Section~\ref{sec:concordance},
we prove Theorem~\ref{thm:intro-inv}, i.e. that the extended signature and nullity are invariant under topological concordance.
Finally, Section~\ref{sec:ex} deals with the aforementioned application to the links~$L(n)$ in Figure~\ref{fig:brunnian_links}.
	
	\subsection*{Acknowledgments}
The authors thank Anthony Conway and Jean-Baptiste Meilhan for useful discussions.
DC and LF are supported by the Swiss NSF grant 200021-212085.

	\section{Background on extended signatures and nullities}
	\label{sec:background}
	
The aim of this first section is to recall the necessary background material on extended signatures and nullities.
It contains no original result, but does gather several easy lemmas that will be needed in this work.
More precisely, we start in Section~\ref{sub:twisted} with the definition of twisted (co)homology and twisted
intersection forms. Section~\ref{sub:ML} deals with a closed~3-manifold~$M_L$ associated to an arbitrary colored
link~$L$. This manifold is necessary for the definition of the extended signature and nullity of~$L$,
which is stated in Section~\ref{sub:ext} following~\cite{CMP23}. Finally, Section~\ref{sub:NW} recalls the Novikov-Wall theorem for the
non-additivity of the (twisted) signature. 
	
	\subsection{Twisted coefficients}
	\label{sub:twisted}
	
	In this subsection, we briefly recall the definition of (co)homology with twisted coefficients and of the corresponding intersection form, referring to~\cite[Part~XXIX]{Friedl} for details.
	
\medskip

Let~$G$ be a group and~$M$ be a left~$\Z[G]$-module. We denote by~$\overline{M}$ the right~$\Z[G]$-module
with the same underlying abelian group as~$M$, but with the action of~$\Z[G]$ induced by~$m\cdot g:=g^{-1}\cdot m$
for all~$g\in G$ and~$m\in M$.

Fix a finite connected pointed CW-complex~$(X,x_0)$ with a (possibly empty) subcomplex~$Y\subset X$,
and let $p\colon\widetilde{X}\to X$ be the universal cover of~$X$.
Recall that the action of~$\pi:=\pi_1(X,x_0)$ equips the cellular chain complex~$C_*(\widetilde{X},p^{-1}(Y))$ with the
structure of a left~$\Z[\pi]$-module.
Fix a ring~$R$ and a homomorphism $\phi\colon\Z[\pi]\to R$.
The map~$\phi$ endows~$R$ with compatible left module structures over the rings~$\Z[\pi]$ and~$R$ (a so-called~$(R,\Z[\pi])$-left left module structure~\cite{Friedl}), which we denote by~$M$.
The associated module~$\overline{M}$ is then an~$(R,\Z[\pi])$-bimodule, and we can consider the chain and cochain complexes of left~$R$-modules
\begin{align*}
C_*(X,Y;M)&=\overline{M}\otimes_{\Z[\pi]}C_*(\widetilde{X},p^{-1}(Y))\,,\\
C^*(X,Y;M)&:=\operatorname{Hom}_{\mathrm{left-}\Z[\pi]}(C_*(\widetilde{X},p^{-1}(Y)),M)\,.
\end{align*}
The homology (resp. cohomology) of the above chain (resp. cochain)
complex is called the {\em twisted homology\/} (resp. {\em twisted cohomology\/}) of~$(X,Y)$,
and is denoted by~$H_*(X,Y;M)$ (resp.~$H^*(X,Y;M)$). Note that these groups are left~$R$-modules.

In the present work, we only use rather specific examples of such twisted coefficients, that we now describe using the
notations introduced above.

\begin{exs}
\label{exs:twisted}
\begin{enumerate}
\item\label{exs:twisted-1} Assume that the CW-complex~$X$ is endowed with a group homomorphism
\[
\pi=\pi_1(X,x_0)\stackrel{\varphi}{\longrightarrow}\Z^\mu=\left<t_1,\dots,t_\mu\right>\,.
\]
This induces a ring homomorphism~$\phi\colon\Z[\pi]\to\Lambda$, where
\[
\Lambda:=\Z[\Z^\mu]=\Z[t_1^{\pm 1},\dots,t_\mu^{\pm 1}]\,.
\]
Using the same symbol~$\Lambda$ for the ring~$R$ and for the module~$M$, we
can consider the twisted homology groups~$H_*(X;\Lambda)$, which are~$\Lambda$-modules.

Assuming that~$\varphi\colon\pi\to\Z^\mu$ is onto, it defines a regular~$\Z^\mu$ cover~$\widetilde{X}^\varphi\to X$,
and one can check that there are natural isomorphisms of~$\Lambda$-modules between the
twisted homology~$H_*(X;\Lambda)$ and the untwisted homology~$H_*(\widetilde{X}_\varphi;\Z)$ (see~\cite[Proposition~239.2]{Friedl}).

\item\label{item:Comega} Assume once again that~$X$ admits a homomorphism~$\varphi\colon\pi\to\Z^\mu$, and fix~$\omega=(\omega_1,\dots,\omega_\mu)\in\T^\mu\coloneq(S^1)^\mu$. Set~$R=\C$ and consider the ring homomorphism~$\phi\colon\Z[\pi]\to\C$ given by the
composition
\[
\Z[\pi]{\longrightarrow}\Lambda{\longrightarrow}\C\,,
\]
where the first map is induced by~$\varphi$ and the second one maps~$t_i$ to~$\omega_i$. Writing~$\C^\omega$ for the resulting module~$M$, this yields twisted (co)homology groups~$H_*(X,Y;\C^\omega)$ and~$H^*(X,Y;\C^\omega)$, which are complex vector spaces. Note that this notation can be slightly misleading, as these vector spaces depend on~$\omega$, but also on~$\varphi$.

If~$\omega=(1,\dots,1)$, then the chain complex of~$\C$-vector spaces~$C_*(X,Y;\C^\omega)$
is naturally isomorphic to the (untwisted) chain complex~$C_*(X,Y;\C)$, yielding~$H_*(X,Y;\C^\omega)\simeq H_*(X,Y;\C)$ in that case. Similarly, we have natural isomorphisms~$H^*(X,Y;\C^\omega)\simeq H^*(X,Y;\C)$.
\end{enumerate}
\end{exs}

We now come to the twisted intersection form, focusing on the setting of Example~\ref{exs:twisted}\ref{item:Comega} above, and 
referring to~\cite[Chapter~244]{Friedl} for more details and full generality.

Assume that~$X$ is a compact oriented $4$-manifold, possibly with boundary, endowed with a group homomorphism~$\varphi\colon\pi\to\Z^\mu$, and fix~$\omega=(\omega_1,\dots,\omega_\mu)\in\T^\mu$. Consider the composition
\begin{equation}
	\label{eq:Q}
	H_2(X;\C^\omega)\longrightarrow H_2(X,\partial X;\C^\omega)\stackrel{\mathrm{PD}}{\longrightarrow} H^2(X;\C^\omega)\stackrel{\mathrm{ev}}{\longrightarrow} \operatorname{Hom}_\C(H_2(X;\C^\omega),\C)\,,
\end{equation}
where the first map is induced by the inclusion~$(X,\emptyset)\subset(X,\partial X)$, the second is the twisted Poincaré duality isomorphism, and the last one is an evaluation map which in the present case is also an
isomorphism (see e.g.~\cite[Proposition~2.3]{CNT}). This is the adjoint map of a Hermitian form
\[
Q_X^{\omega}\colon H_2(X;\C^\omega)\times H_2(X;\C^\omega)\longrightarrow\C
\]
which is called the~$\C^\omega${\em-twisted intersection form\/} on~$X$. One can then write
\[
\sigma_\omega(X):=\sign(Q_X^\omega)\quad\text{and}\quad\eta_\omega(X):=\operatorname{null}(Q_X^\omega)
\]
for the signature and nullity of the~$\C^\omega$-twisted intersection form on~$X$.

\begin{rems}
	\label{rems:int}
	\begin{enumerate}
		\item\label{rems:int-1} If~$\omega=(1,\dots,1)$, then
		we obtain the untwisted signature and nullity of~$X$, denoted by~$\sigma(X)$ and~$\eta(X)$: this follows from
		the discussion in Example~\ref{exs:twisted}\ref{item:Comega}.
		\item\label{rems:int-2} If~$-X$ denotes the manifold~$X$ endowed with the opposite orientation, then we have
		the equality~$\sigma_\omega(-X)=-\sigma_\omega(X)$ for all~$\omega\in\mathbb{T}^\mu$: this follows from the definition of the Poincaré duality isomorphism.
		\item\label{rems:int-3} If the~4-manifold is such that the inclusion induced homomorphism~$H_2(\partial X;\C^\omega)\to H_2(X;\C^\omega)$ is onto, then~$Q^\omega_X$ vanishes identically: this follows from the definition of~$Q^\omega_X$ via~\eqref{eq:Q}. In particular, we then have~$\sigma_\omega(X)=0$.
	\end{enumerate}
\end{rems}

	\subsection{The generalized Seifert surgery}
	\label{sub:ML}
	
	Given a~$\mu$-colored link~$L=L_1\cup\dots\cup L_\mu$, one can consider its exterior~$X_L=S^3\setminus\nu(L)$ equipped with
	the natural homomorphism~$\varphi_X\colon\pi_1(X_L)\to\Z^\mu$ induced by the coloring. The aim of this paragraph is to recall the construction of an oriented $3$-manifold~$P_L$, which only depends on~$L$ and can
	be glued to~$X_L$ so that the resulting oriented {\em closed\/} manifold~$M_L=X_L\cup_\partial -P_L$ admits a homomorphism~$\varphi\colon\pi_1(M_L)\to\Z^\mu$ extending~$\varphi_X$.
	This is the mild variation on~\cite[Construction~4.17]{Tof20} presented in~\cite{CMP23}, see also~\cite{CNT}.
	
	\medskip
	
	Given a~$\mu$-colored link~$L$, 
	consider the decorated graph~$\Gamma_L$ defined as follows:
	\begin{itemize}
	\item 
	The vertices of~$\Gamma_L$ are indexed by the components of~$L$, and each vertex~$K$ is decorated with an oriented closed disc~$D_K$.
	\item 
	Given any two components~$K,K'$ of different colors, the corresponding vertices are linked by~$\vert\lk(K,K')\vert$ edges, and every such edge~$e$ is decorated with the sign~$\varepsilon(e)=\sgn(\lk(K,K'))$.
	\end{itemize}

 The~3-manifold~$P_L$ is the so-called {\em plumbed manifold\/} constructed from~$\Gamma_L$ in the following standard manner. For any component~$K\subset L$, let~$D_K^\circ$ be obtained from~$D_K$ by removing disjoint open~$2$-discs~$D_{K,e}$ indexed by edges~$e$ adjacent to~$K$. Set
 \[
 P_L=\bigsqcup_{K\subset L}D_K^\circ\times S^1/\sim\,,
 \]
 where~$S^1$ denotes an oriented circle and for each edge~$e$, we have performed the gluing
 \begin{equation}
  \label{eq:glue}
  \begin{aligned}
 (-\partial D_{K,e})\times S^1 &\longleftrightarrow (-\partial D_{K',e})\times S^1\\
 (x,y)&\longleftrightarrow (y^{-\varepsilon(e)},x^{-\varepsilon(e)})
 \end{aligned}
 \end{equation}
with~$K$ and~$K'$ the two vertices linked by~$e$. Note that the orientations of~$D_K$ and~$S^1$
induce an orientation of~$D_K^\circ\times S^1$. Since the identifications~\eqref{eq:glue} make use of orientation
reversing homeomorphisms, these orientations extend to an orientation of~$P_L$.

The resulting oriented compact $3$-manifold~$P_L$ has boundary~$\partial P_L=\bigsqcup_{K\subset L} \partial D_K\times S^1$. Therefore, it is possible to glue~$P_L$ and~$X_L$ along their boundaries, and we do so as follows. For each component~$K\subset L$, a {\em meridian\/} of~$K$ is an oriented simple closed curve~$m_K\subset \partial\nu(K)$ such that~$[m_K]=0\in H_1(\nu(K))$ and~$\lk(m_K,K)=1$.
Also, a {\em Seifert longitude\/} of~$K\subset L_i$ is an oriented simple closed curve~$\ell_K\subset \partial\nu(K)$ such that~$[\ell_K]=[K]\in H_1(\nu(K))$ and
\begin{equation}
\label{eq:Seifert}
\lk(\ell_K,L_i):=\sum_{K'\subset L_i}\lk(\ell_K,K')=0\,.
\end{equation}
We glue~$X_L$ and~$P_L$ along their boundary via the homeomorphism~$\partial D_K\times S^1\simeq \partial\nu(K)$ obtained by mapping~$\ast_K\times S^1$ (for some~$\ast_K\in\partial D_K$) to a meridian~$m_K$ and~$\partial D_K\times\ast$ (for some~$\ast\in S^1$) to a Seifert longitude~$\ell_K$.
Since the orientations on~$X_L$ and on~$P_L$ induce the same orientation on the boundary tori, we reverse the orientation of~$P_L$ and define
\[
M_L:=X_L\cup_\partial -P_L\,.
\]
This is an oriented closed~3-manifold called the {\em generalized Seifert surgery\/} on the colored link~$L$.

The main point of this construction is the following fact, see~\cite[Lemma~2.11]{CMP23}:
the homomorphism~$\varphi_X\colon\pi_1(X_L)\to\Z^\mu$ defined by~$\varphi_X([\gamma])=\left(\lk(\gamma,L_i)\right)_i$ extends to~$\varphi\colon\pi_1(M_L)\to\Z^\mu$ such that~$\varphi([\ast_i\times S^1])=t_i\in\Z^\mu$ for any~$\ast_i\in D_K$ with~$K\subset L_i$.
Note however that this extension is in general not unique.
We shall call such a map~$\varphi\colon\pi_1(M_L)\to\Z^\mu$
a {\em meridional\/} homomorphism.

\begin{exs}
\label{exs:ML}
\begin{enumerate}
\item\label{exs:ML-1} If~$L$ is a~$\mu$-component~$\mu$-colored link with all linking numbers
vanishing (this includes the case of knots), then~$M_L$ is the~$0$-surgery on~$L$, and~$\varphi\colon H_1(M_L)\to\Z^\mu$ the unique extension of the isomorphism~$H_1(X_L)\simeq\Z^\mu$.
For example, if~$L$ is the $3$-colored Borromean rings~$B$, then~$M_B$ is the $3$-dimensional torus and~$\varphi\colon H_1(S^1\times S^1\times S^1)\to\Z^3$ the isomorphism determined by the orientations and colors of the components of~$B$.
\item\label{exs:ML-2} If~$L$ is an oriented link (interpreted as a~1-colored link), then~$M_L$ is the {\em Seifert sugery\/}
on~$L$ as defined in~\cite[Definition~5.1]{NP17}. This justifies the terminology.
\end{enumerate}
\end{exs}

	\subsection{Extended signatures  and nullities}
	\label{sub:ext}

We now recall the construction of the extended signature and nullity functions, following~\cite{CMP23}
and referring to it for details.

Let~$L$ be a~$\mu$-colored link, and let~$M_L$ be the associated generalized Seifert surgery defined in Section~\ref{sub:ML}. As recalled above, the natural homomorphism~$H_1(X_L)\to\Z^\mu$ defined by the orientation and colors of the components of~$L$ extends to a (non-canonical) meridional homomorphism~$\varphi\colon H_1(M_L)\to\Z^\mu$, thus defining an element~$(M_L,\varphi)$ in the bordism group~$\Omega_3(\Z^\mu)$. Via
the well-known isomorphism~$\Omega_3(\Z^\mu)\simeq H_3(\mathbb{T}^\mu)=\Z^{\mu\choose 3}$
(see e.g.~\cite[Section~3]{DNOP}), we get a family of integers
\[
\mu_L=\{\mu_L(ijk)\mid 1\le i<j<k\le \mu\}\in\Z^{\mu\choose 3}
\]
which depend on~$L$, but also on the choice of the meridional homomorphism~$\varphi$. For example, if~$B(123)$
denotes the $3$-colored Borromean rings appropriately orientated, then~$\mu_{B(123)}(123)=1$: this follows
from Example~\ref{exs:ML}\ref{exs:ML-1} and the explicit form of the isomorphism~$\Omega_3(\Z^\mu)\simeq\Z^{\mu\choose 3}$.

Now, consider the auxiliary~$\mu$-colored link
\begin{equation}
\label{eq:Lsharp}
L^\#:= L\sqcup \bigsqcup_{i<j<k} -\mu_L(ijk)\cdot B(ijk)\,,
\end{equation}
where~$\sqcup$ denotes the distant sum of links,~$B(ijk)$ is the Borromean rings oriented and colored so that~$\mu_{B(ijk)}(ijk)=1$, and~$n\cdot B(ijk)$ stands for the distant sum of~$\vert n\vert$ copies of~$B(ijk)$ (resp. of~$B(jik)$) if~$n\ge 0$ (resp. if~$n\le 0$). The homomorphism~$\varphi\colon H_1(M_L)\to\Z^\mu$ extends uniquely to~$\varphi^\#\colon H_1(M_{L^\#})\to\Z^\mu$ which by construction satisfies~$(M_{L^\#},\varphi^\#)=0\in\Omega_3(\Z^\mu)$.

Next, consider a bounding surface~$F\subset B_4$ for~$L^\#$, i.e. a collection~$F_1\cup\dots\cup F_\mu$ of
locally flat surfaces properly embedded in~$B^4$ that only intersect each other transversally in double points and such that~$\partial F_i=L^\#_i\subset S^3=\partial B^4$ for all~$i$. For a well-chosen~$F$, its exterior~$V_F:=B^4\setminus\nu(F)$ satisfies~$\pi_1(V_F)\simeq\Z^\mu$.
Moreover, its boundary splits as~$\partial V_F=X_{L^\#}\cup-P_F$, with~$P_F$ a plumbed manifold associated with~$F$. The restriction~$H_1(P_F)\to\Z^\mu$ of the isomorphism~$H_1(V_F)\simeq\Z^\mu$ and the restriction~$H_1(P_{L^\#})\to\Z^\mu$ of~$\varphi^\#$ define a meridional homomorphism~$\psi\colon H_1(P_F\cup-P_{L^\#})\to\Z^\mu$. One of the most technical results of~\cite{CMP23}, namely its Lemma~2.14, asserts that there exists an oriented compact $4$-manifold~$Y_F$ such that~$\partial Y_F=P_F\cup-P_{L^\#}$, and an isomorphism~$\Psi\colon\pi_1(Y_F)\stackrel{\simeq}{\to}\Z^\mu$ which extends~$\psi$,
and such that~$\sigma_\omega(Y_F)$ vanishes for all~$\omega\in\T^\mu$. Hence, one can consider the oriented compact $4$-manifold
\[
W_F=V_F\cup_{P_F} Y_F\,,
\]
which is endowed with an isomorphism~$\Phi\colon \pi_1(W_F)\stackrel{\simeq}{\to}\Z^\mu$. By construction, we
have~$\partial W_F=M_{L^\#}$, and~$\Phi$ extends~$\varphi^\#$. This is illustrated in Figure~\ref{fig:WF}.

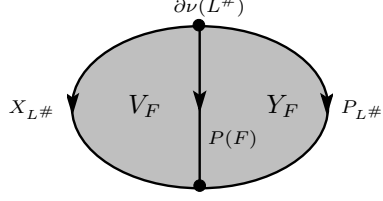
\begin{figure}[tbp]
		\centering
		\begin{overpic}[width=3.5cm]{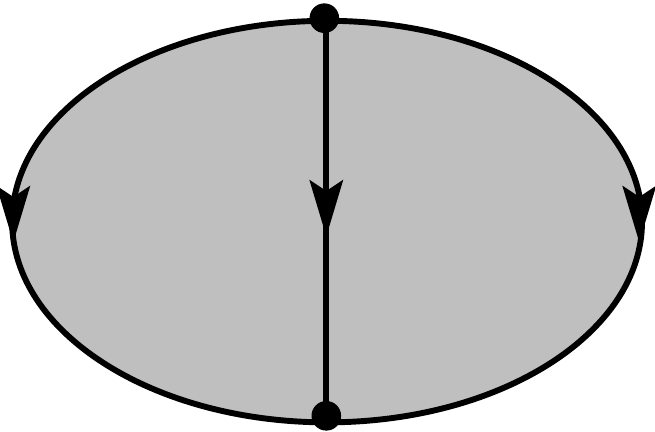}
			\put (-22,30){\scriptsize{$X_{L^\#}$}}
			\put (53,18){\scriptsize$P(F)$}
			\put (103,30){\scriptsize$P_{L^\#}$}
			\put (23,30){$V_F$}
			\put (75,30){$Y_F$}
			\put (40,68){\scriptsize$\partial\nu(L^\#)$}
		\end{overpic}
		\caption{Construction of the manifold~$W_F$.}
		\label{fig:WF}
	\end{figure}
	
	\begin{definition}
	\label{def:sign-null}
	The {\em (extended) signature\/} and {\em nullity\/} of the~$\mu$-colored link~$L$ are the maps
	\[
	\sigma_L,\eta_L\colon\T^\mu\longrightarrow\Z
	\]
	defined
	by~$\sigma_L(\omega)=\sigma_\omega(W_F)$ and~$\eta_L(\omega)=\eta_\omega(W_F)-\sum_{i<j<k}\vert\mu_L(ijk)\vert-2\sum_{\stackrel{i<j<k}{\omega_i=\omega_j=\omega_k=1}}\vert\mu_L(ijk)\vert$.
	\end{definition}
	
	In~\cite[Theorem~4.4]{CMP23}, it is checked that the maps~$\sigma_L$ and~$\eta_L$ are well-defined
	invariants of~$L$.
	Moreover, they extend
	the multivariable signature and nullity previously defined on the open torus~$(S^1\setminus\{1\})^\mu$
	using generalized Seifert surfaces~\cite{CF08}.
	
	\begin{rem}
	\label{rem:rho}
	Since the manifold~$(M_{L^\#},\varphi^\#)$ bounds over~$\Z^\mu$, one could have simply considered the signature defect~$\sigma_\omega(W)-\sigma(W)$, with~$(W,\Phi)$ any oriented compact $4$-manifold with
	boundary~$(M_{L^\#},\varphi^\#)$. This is an invariant, known as the (opposite of the)~{\em$\rho$-invariant\/} of~$M_{L^\#}$, originally defined by Atiyah, Patodi and Singer in~\cite{APS75-I,APS75-II}. Indeed, if~$W$ is any
	oriented compact connected $4$-manifold endowed with a map~$\alpha\colon\pi_1(W)\to S^1$, then
	\begin{equation}
	\label{eq:rho-W}
	\rho(\partial W,\alpha\circ i_*)=\sigma(W)-\sigma_\alpha(W)\,,
	\end{equation}
	with~$i_*\colon\pi_1(\partial W)\to\pi_1(W)$ the inclusion induced homomorphism,
	and~$\sigma_\alpha(W)$ the signature of~$W$ with coefficients twisted by~$\alpha$, see~\cite[Theorem~2.4]{APS75-II}.
	For all~$\omega\in\T^\mu$, standard properties of the $\rho$-invariant then lead to the equality
	\[
	\sigma(W)-\sigma_\omega(W)=\rho(M_{L^\#},\chi_\omega\circ\varphi^\#)=\rho(M_{L},\chi_\omega\circ\varphi)\,,
	\]
	where~$\chi_\omega\colon\Z^\mu\to S^1$ denotes the homomorphism determined by~$t_i\mapsto\omega_i$,
	see the proof of~\cite[Theorem~4.4]{CMP23}.
	The issue is that, in general, this invariant does {\em not\/} extend the usual multivariable signature. This is the reason why we have to consider the auxiliary link~$L^\#$ together with the specific $4$-manifold~$W_F$, and set~$\sigma_L(\omega)=\sigma_\omega(W_F)$. In conclusion, we have the equality
	\begin{equation}
	\label{eq:rho}
	\sigma_L(\omega)=\sigma_\omega(W_F)=\sigma(W_F)-\rho(M_{L},\chi_\omega\circ\varphi)\,.
	\end{equation}
	Note that the difference between these two invariants, namely the untwisted signature~$\sigma(W_F)$, only depends on the
	linking numbers and colors of the components of~$L$ (see the proof of~\cite[Lemma~A.6]{CMP23}).
	\end{rem}
	
	We will need the following easy result, which extends~\cite[Proposition~2.10]{CF08}.
	
	\begin{prop}
	\label{prop:sign-L}
For any~$\mu$-colored link~$L$, the extended signature of  the mirror image~$-L$ of~$L$ satisfies~$\sigma_{-L}=-\sigma_L$.
	\end{prop}
	\begin{proof}
	For each~$K\subset L$, consider the orientation reversing homeomorphism~$D_K\times S^1\to D_K\times S^1$ defined by~$(x,y)\mapsto (x,y^{-1})$. Since the graph~$\Gamma_{-L}$ associated to~$-L$ is obtained from~$\Gamma_L$ by reversing the sign~$\varepsilon(e)$ of all edges~$e$, these homeomorphisms are coherent 
	with the gluing~\eqref{eq:glue} and define an orientation reversing homeomorphism~$P_L\to P_{-L}$.
By definition of~$-L$, we also have an orientation reversing homeomorphism~$X_L\to X_{-L}$.
Together, they define an orientation reversing homeomorphism~$M_L\to M_{-L}$, the existence of which can be stated by the equality
\begin{equation}
\label{eq:-ML}
M_{-L}=-M_L\,.
\end{equation}
Chosing the same meridional homomorphism~$\varphi$ on~$H_1(M_{-L})=H_1(M_L)$, and using the fact that the Borromean rings are amphicheiral, we get~$\mu_{-L}=\mu_L$ and~$(-L)^\#=-(L^\#)$. By~\eqref{eq:-ML}, this yields~$M_{(-L)^\#}=-M_{L^\#}$, endowed with the same homomorphism~$\varphi^\#$ on~$H_1(M_{(-L)^\#})=H_1(M_{L^\#})$.
Hence, if~$W_F=V_F\cup Y_F$ endowed with~$\Phi\colon\pi_1(W_F)\stackrel{\simeq}{\to}\Z^\mu$ is an oriented $4$-manifold over~$\Z^\mu$ that can be used to define~$\sigma_L$, then the oriented manifold~$-W_F=-V_F\cup -Y_F$ endowed with the same~$\Phi$ on~$\pi_1(-W_F)=\pi_1(W_F)$ can be used to define~$\sigma_{-L}$. By Remark~\ref{rems:int}\ref{rems:int-2}, we now have
\[
\sigma_{-L}(\omega)=\sigma_\omega(-W_F)=-\sigma_\omega(W_F)=-\sigma_L(\omega)
\]
for all~$\omega\in\T^\mu$.
	\end{proof}
	
	\subsection{The Novikov-Wall theorem}
	\label{sub:NW}
	
	In this final background section, we recall the {\em Novikov-Wall theorem\/}~\cite{Wal} and present an algebraic lemma, both of which are needed in the proof of Theorem~\ref{thm:conc}.
	
	\medskip
	
	Let~$W$ be an oriented compact~4-manifold,
	and let~$X_0$ be an oriented compact~3-manifold properly embedded in~$W$. Assume that~$X_0$ splits~$W$ into two manifolds~$W_-$
	and~$W_+$, with~$W_-$ such that the induced orientation on its boundary restricted to~$X_0\subset\partial W_-$
	coincides with the given orientation of~$X_0$.
	For~$\varepsilon=\pm$, let~$X_\varepsilon$ denote the~3-manifold~$\partial W_\varepsilon\setminus\mathrm{int}(X_0)$ endowed with the orientation so that~$\partial W_-=(-X_-)\cup X_0$ and~$\partial W_+=(-X_0)\cup X_+$, see Figure~\ref{fig:NW}. Note that the orientations of~$X_0$,~$X_-$ and~$X_+$ induce the same orientation on the
	surface~$\Sigma:=\partial X_0=\partial X_-=\partial X_+$. We will use this orientation to define the (twisted) intersection form on this surface.
	
	\begin{figure}[tbp]
		\centering
		\begin{overpic}[width=3.5cm]{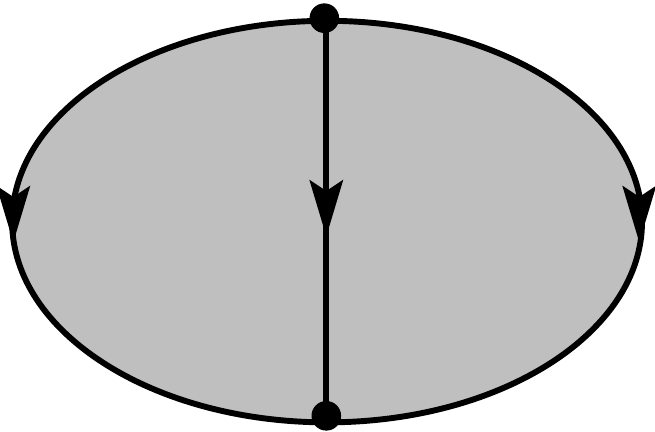}
			\put (-18,30){$X_-$}
			\put (55,25){$X_0$}
			\put (103,30){$X_+$}
			\put (23,30){$W_-$}
			\put (75,30){$W_+$}
			\put (45,68){$\Sigma$}
		\end{overpic}
		\caption{The decomposition~$W=W_-\cup_{X_0} W_+$ in the Novikov-Wall theorem.}
		\label{fig:NW}
	\end{figure}
	
	Assume that~$W$ is endowed with a homomorphism~$\psi\colon\pi_1(W)\to\Z^\mu$.
	As described in Example~\ref{exs:twisted}\ref{item:Comega}, any~$\omega\in\mathbb{T}^\mu$ induces twisted coefficients~$\mathbb{C}^\omega$ on the homology of~$W$.
	Precomposing~$\psi$ with inclusion induced homomorphisms,
	we also get twisted coefficients on the homology of submanifolds of~$W$, coefficients that we denote by~$\mathbb{C}^\omega$ as well. Note however that these submanifolds need not be connected, so one needs to be careful
	of the meaning of these twisted homology spaces; in the terminology of~\cite[Chapter~240]{Friedl}, these
	are so-called {\em internal twisted homology\/} spaces. The point is that the long exact sequence of the pair~$(X_\varepsilon,\Sigma)$ holds for any~$\varepsilon\in\{-,0,+\}$. Using this exact sequence together with Poincar\'e-Lefschetz duality~\cite[Theorem~243.1]{Friedl} and the Universal Coefficient Theorem in the form of~\cite[Proposition~2.3]{CNT}, one easily checks that for any~$\varepsilon\in\{-,0,+\}$, the
	kernel~$\mathcal{L}_\varepsilon$
	of the inclusion induced homomorphism~$H_1(\Sigma;\mathbb{C}^\omega)\to H_1(X_\varepsilon;\mathbb{C}^\omega)$ is a Lagrangian subspace of~$(H_1(\Sigma;\mathbb{C}^\omega),\,\cdot\,)$, where~$(a,b)\mapsto a\cdot b$ denotes the (non-degenerate and skew-Hermitian) twisted intersection form on~$H_1(\Sigma;\mathbb{C}^\omega)$.
	
	Given three Lagrangian subspaces~$\mathcal{L}_-,\mathcal{L}_0,\mathcal{L}_+$ of a finite-dimensional complex vector space~$H$
	endowed with a non-degenerate skew-Hermitian form~$(a,b)\mapsto a\cdot b$, the associated {\em Maslov index} is the integer
	\[
	\mathit{Maslov}(\mathcal{L}_-,\mathcal{L}_0,\mathcal{L}_+)=\sign(f)\,,
	\]
	where~$f$ is the following Hermitian form on~$(\mathcal{L}_-+\mathcal{L}_0)\cap\mathcal{L}_+$. 
	Given~$a,b\in(\mathcal{L}_-+\mathcal{L}_0)\cap\mathcal{L}_+$,
	write~$a=a_-+a_0$ with~$a_-\in\mathcal{L}_-$ and~$a_0\in\mathcal{L}_0$ and set~$f(a,b):=a_0\cdot b$.
	
	The Novikov-Wall theorem states that, in the setting above and for any~$\omega\in\mathbb{T}^\mu$, we have the equality
	\begin{equation}
		\label{eq:NW}
		\sign_\omega(W)=\sign_\omega(W_-)+\sign_\omega(W_+)+\mathit{Maslov}(\mathcal{L}_-,\mathcal{L}_0,\mathcal{L}_+)\,.
	\end{equation}
	Note that this theorem was originally stated and proved by Wall~\cite{Wal} in the untwisted case, but the proof readily extends. Note also that the above version follows the conventions of~\cite[Chapter~IV.3]{Tur}, which slightly differ from the original ones.
	
	\medskip
	
	In addition to the Novikov-Wall theorem, we shall also need the following algebraic lemma.
	
	\begin{lem}
		\label{lem:alg}
		Let~$\phi\colon(H,\,\cdot\,)\to(H',\,\cdot'\,)$ be an isometry between vector spaces endowed with non-degenerate skew-Hermitian forms,
		and let~$V_-,V_+\subset H$ be Lagrangians. Then, the three subspaces of~$H\oplus H'$ given by
		\[
		\mathcal{L}_-:=V_-\oplus\phi(V_-)\,,\quad \mathcal{L}_0:=\{(x,-\phi(x))\mid x\in H\}\,,\quad\text{and}\quad \mathcal{L}_+:=V_+\oplus\phi(V_+)
		\]
		are Lagrangians with respect to the non-degenerate skew-Hermitian form~$(x,x')\bullet (y,y')=x\cdot y-x'\cdot'y'$ on~$H\oplus H'$.
		Furthermore, they satisfy~$\mathit{Maslov}(\mathcal{L}_-,\mathcal{L}_0,\mathcal{L}_+)=0$.
	\end{lem}
	\begin{proof}
		To show that~$\mathcal{L}_-$ is a Lagrangian subspace of~$(H\oplus H',\bullet)$, consider
		elements~$(x_-,\phi(y_-))$ and~$(w_-,\phi(z_-))$ of~$\mathcal{L}_-$, with~$x_-,y_-,w_-,z_-\in V_-$. Since~$\phi$ is
		an isometry and~$V_-$ is an isotropic subspace of~$(H,\,\cdot\,)$, we have
		\[
		(x_-,\phi(y_-))\bullet (w_-,\phi(z_-))=x_-\cdot w_--\phi(y_-)\cdot'\phi(z_-)=x_-\cdot w_--y_-\cdot z_-=0-0=0\,,
		\]
		showing that~$\mathcal{L}_-$ is isotropic. Since~$\dim(\mathcal{L}_-)=2\dim(V_-)=\dim(H)=\frac{1}{2}\dim(H\oplus H')$
		and the form~$\bullet$ is non-degenerate, it follows that~$\mathcal{L}_-$ is Lagrangian. The proof for~$\mathcal{L}_+$ is identical. To check that~$\mathcal{L}_0$ is Lagrangian, fix elements~$(x,-\phi(x))$ and~$(y,-\phi(y))$ of~$\mathcal{L}_0$, with~$x,y\in H$. Since~$\phi$ is an isometry, we have
		\[
		(x,-\phi(x))\bullet (y,-\phi(y))=x\cdot y-\phi(x)\cdot'\phi(y)=0\,,
		\]
		implying that~$\mathcal{L}_0$ is Lagrangian since its dimension is half that of~$H\oplus H'$.
		
		Recall that~$\mathit{Maslov}(\mathcal{L}_-,\mathcal{L}_0,\mathcal{L}_+)$ is the signature of the
		form~$f$ on~$(\mathcal{L}_-+\mathcal{L}_0)\cap\mathcal{L}_+$ defined by~$f(a,b)=a_0\bullet b$ for~$a=a_-+a_0\in(\mathcal{L}_-+\mathcal{L}_0)\cap\mathcal{L}_+$ with~$a_-\in\mathcal{L}_-$ and~$a_0\in\mathcal{L}_0$.
		We will now show that this form~$f$ is identically zero, implying the statement. An element~$a=a_-+a_0$ as above is of the form~$a=(x_+,\phi(y_+))\in\mathcal{L}_+$ with~$x_+,y_+\in V_+$, and there exist~$x_-,y_-\in V_-$ such that~$a_-:=(x_-,\phi(y_-))\in\mathcal{L}_-$ satisfies~$a_0:=a-a_-\in\mathcal{L}_0$. Hence, there exists~$x\in H$ such that
		\[
		a_0=a_+-a_-=(x_+-x_-,\phi(y_+-y_-))=(x,-\phi(x))\,.
		\]
		Therefore, we have~$x=x_+-x_-=y_--y_+$, leading to~$x_-+y_-=x_++y_+\in V_-\cap V_+$.
		Similarly, an element~$b\in(\mathcal{L}_-+\mathcal{L}_0)\cap\mathcal{L}_+$ is given by~$b=(w_+,\phi(z_+))$ with~$w_+,z_+\in V_+$ and~$w_++z_+\in V_-\cap V_+$. Hence, we have
		\begin{align*}
			f(a,b)&=a_0\bullet b=(x_+-x_-,-\phi(x_+-x_-))\bullet (w_+,\phi(z_+))=(x_+-x_-)\cdot w_++\phi(x_+-x_-)\cdot'\phi(z_+)\\
			&=(x_+-x_-)\cdot w_++(x_+-x_-)\cdot z_+=\underbrace{x_+\cdot w_+}_{=0}+\underbrace{x_+\cdot z_+}_{=0}-\underbrace{x_-}_{\in V_-}\cdot \underbrace{(w_++z_+)}_{\in V_-}=0\,,
		\end{align*}
		where we used the facts that~$\phi$ is an isometry and that~$V_-,V_+$ are isotropic.
	\end{proof}

	\section{multivariable Hosokawa polynomials}
	\label{sec:Hosokawa}
	
The aim of this section is to show that the manifold~$M_L$ yields renormalizations of the Alexander module and polynomial of a colored link, and to compute this renormalized polynomial explicitly. As we shall see, the result can be understood as a multivariable generalization of the Hosokawa polynomial defined in~\cite{Hos}.
This polynomial will play a crucial role in the understanding of the extended signatures, see Theorem~\ref{thm:main}.
	
	\medskip
	
We start with the definitions of these objects.
Let~$L=L_1\cup\dots\cup L_\mu$ be a~$\mu$-colored link. As usual, we denote by~$X_L$ its exterior, and by~$\varphi_X\colon\pi_1(X)\to\Z^\mu$ the (surjective) homomorphism defined by~$\varphi_X([\gamma])=\left(\lk(\gamma,L_i)\right)_i$. By Example~\ref{exs:twisted}\ref{exs:twisted-1}, one can consider the associated~$\Lambda$-module~$H_1(X_L;\Lambda)$, which is called the {\em Alexander module\/} of the colored link~$L$. Note that,
by the discussion in this same example, this module can also be defined as the untwisted homology of the~$\Z^\mu$-cover of~$X_L$ associated to the homomorphism~$\varphi_X$.
The {\em Alexander polynomial\/} of the  colored link~$L$ is a greatest common divisor~$\Delta_L\in\Lambda$
	of the first elementary ideal of~$H_1(X_L;\Lambda)$. Obviously, it is only well-defined up to multiplication
	by units of~$\Lambda$, i.e. a sign and powers of the~$t_i$'s. We will write~$\Delta\stackrel{\cdot}{=}\Delta'$
	for an equality in~$\Lambda$ up to multiplication by units.

By~\cite[Lemma~2.11]{CMP23}, the homomorphism~$\varphi_X$ extends to~$\varphi\colon\pi_1(M_L)\to\Z^\mu$, allowing us to play the same game with the closed $3$-manifold~$M_L=X_L\cup_\partial -P_L$.
	
	\begin{definition}
	\label{def:Hosokawa}
	The {\em Hosokawa polynomial\/} of~$L$ is a greatest common divisor~$\widetilde\Delta_L\in\Lambda$
	of the first elementary ideal of the~$\Lambda$-module~$\A(L):=H_1(M_L;\Lambda)$.
	\end{definition}
	
	Several remarks are in order.
	
	\begin{rems}
	\label{rems:Hoso}
	\begin{enumerate}
	\item\label{rems:Hoso-0} In general, the~$\Lambda$-module~$\A(L):=H_1(M_L;\Lambda)$ might depend on the choice of the meridional homomorphism~$\varphi\colon H_1(M_L)\to\Z^\mu$. However, the Hosokawa polynomial
	does not, as we shall see in Proposition~\ref{prop:Hosokawa} below.
	
	\item\label{rems:Hoso-1} If~$L=K$ is a knot, then~$M_L$ is the~0-surgery on~$K$ (recall Example~\ref{exs:ML}\ref{exs:ML-1}). In that case, the Mayer-Vietoris associated with the decomposition~$M_L=X_L\cup -P_L$ easily yields
	an isomorphism between the classical Alexander module~$H_1(X_L;\Lambda)$ and the
	renormalized version~$\A(L)=H_1(M_L;\Lambda)$. (This will be a special case of computations performed in detail below, see Remark~\ref{rems:prop}\ref{rems:prop-1}.) In particular, we have the equality~$\widetilde\Delta_K\stackrel{\cdot}{=}\Delta_K$ for knots,
	and nothing new happens in this case.
	
	\item Consider the localized ring
	\[
	\Lambda_S:=\Z[t_1^{\pm 1},\dots,t_\mu^{\pm 1},(t_1-1)^{-1},\dots,(t_\mu-1)^{-1}]\,.
	\]
	Composing the ring homomorphism~$\Z[\pi]\to\Lambda$ from Example~\ref{exs:twisted}\ref{exs:twisted-1} with the inclusion~$\Lambda\to\Lambda_S$,
	one can consider twisted homology groups~$H_*(M_L;\Lambda_S)$, and similarly for the spaces~$P_L$,~$X_L$
	and their intersection~$\partial X_L$. (This latter space is in general not connected, but this is not an issue,
	see the proof of Proposition~\ref{prop:Hosokawa} below.)
	It is not difficult to check that the spaces~$P_L$ and~$\partial X_L$ are~$\Lambda_S$-acyclic, meaning that
	their homology with twisted coefficients in~$\Lambda_S$ vanish. By the Mayer-Vietoris exact sequence for~$M_L=X_L\cup P_L$, we get an isomorphism
	\[
	H_1(X_L;\Lambda_S)\simeq H_1(M_L;\Lambda_S)\,.
	\]
	Therefore, nothing new happens with coefficients in~$\Lambda_S$ either. For the study of extended signatures however, it is crucial to work over the ring~$\Lambda$, and not over its localization~$\Lambda_S$.
	
	\item The ring~$\Lambda_S$ being a localization of~$\Lambda$, it is a flat~$\Lambda$-module. Therefore,
	the isomorphism displayed above can be expressed as~$\A(L)\otimes_\Lambda\Lambda_S\simeq H_1(X_L;\Lambda)\otimes\Lambda_S$. As a consequence, the Alexander polynomials~$\widetilde\Delta_L$ and~$\Delta_L$ coincide
	up to multiplication by units of~$\Lambda_S$, i.e. by units of~$\Lambda$ and powers of~$(t_i-1)^{-1}$. The main result of this section is an explicit computation of these powers, that we now present.
	\end{enumerate}
	\end{rems}
	
	Given a~$\mu$-colored link~$L=L_1\cup\dots\cup L_\mu$, let us denote by~$\vert L\vert$ the number of components of~$L$, and similarly by~$\vert L_i\vert$ the number of components of~$L_i$.
	
	\begin{prop}
	\label{prop:Hosokawa}
	For any~$\mu$-colored link~$L$, we have
	\[
	\widetilde\Delta_L(t)\stackrel{\cdot}{=}\frac{\Delta_L(t)}{(t-1)^{|L|-1}}
	\]
	if~$\mu=1$, and
	\[
	\widetilde\Delta_L(t_1,\dots,t_\mu)\stackrel{\cdot}{=}
	\prod_{i=1}^\mu (t_i-1)^{\nu_i}\Delta_L(t_1,\dots,t_\mu)\quad\text{with}\quad \nu_i=\Big(\sum_{\substack{K\subset L_i\\ K'\subset L\setminus L_i}}\vert\lk(K,K')\vert\Big)-\vert L_i\vert
	\]
	if~$\mu>1$.
	\end{prop}
	
	We give two disclaimers before starting the proof.
	First, let us point out that we do not claim much originality here, as this proof uses standard techniques, and variations
	on parts of it can be found in the literature (see e.g. the proof of~\cite[Lemma~3.3]{BFP16}).
	Also, these standard techniques being well-known to the experts but rather numerous and sometimes cumbersome, we take the liberty not to
	recall them in a comprehensive background section, but to use them directly in the proof with appropriate references.
	
	\begin{proof}[Proof of Proposition~\ref{prop:Hosokawa}]
	The main idea is to compute the~$\Lambda$-modules appearing in the Mayer-Vietoris exact sequence in twisted homology associated to the decomposition~$M_L=X_L\cup P_L$, and to use Levine's classical result~\cite[Lemma~5]{Lev67} which allows to relate the greatest common divisors of first elementary ideals of modules
	in an exact sequence.
	
	The first step in this plan yields a subtlety: the spaces~$P_L$ and~$X_L\cap P_L=\partial\nu(L)$ are in general not connected,
	while the homology with twisted coefficients is only defined for connected spaces. The idea is to use the fact that these homology modules can also be computed via untwisted homology of covering spaces (recall Example~\ref{exs:twisted}\ref{exs:twisted-1}), and to apply the Mayer-Vietoris exact sequence with untwisted coefficients to these covering spaces. 
	More formally, we shall rely on~\cite[Chapter~240]{Friedl}, but still use the same notation for our twisted homology modules (even though we are actually dealing with so-called {\em internal twisted homology\/}). So, by Theorem~240.7 of~\cite{Friedl}, we have a Mayer-Vietoris exact sequence of~$\Lambda$-modules
	\begin{equation}
  \label{eq:MV}
  \begin{aligned}
	H_2(\partial\nu(L);\Lambda)&\longrightarrow H_2(X_L;\Lambda)\oplus H_2(P_L;\Lambda)\longrightarrow H_2(M_L;\Lambda)\\
	&\longrightarrow H_1(\partial\nu(L);\Lambda)\longrightarrow H_1(X_L;\Lambda)\oplus H_1(P_L;\Lambda)\longrightarrow \A(L)\\
	&\longrightarrow H_0(\partial\nu(L);\Lambda)\longrightarrow H_0(X_L;\Lambda)\oplus H_0(P_L;\Lambda)\longrightarrow H_0(M_L;\Lambda)\longrightarrow 0\,.
	\end{aligned}
	\end{equation}
	
	We start by computing the~$\Lambda$-modules~$H_*(\partial\nu(L);\Lambda)$.
	Let us point out once again that these are in general not twisted homology modules {\em stricto sensu\/}, but that
	they should be understood as~$H_*(p^{-1}(\partial\nu(L))$, with~$p\colon \widetilde{X}^{\varphi}_L\to X_L$
	the~$\Z^\mu$-cover corresponding to the group homomorphism
	\[
	\varphi\colon\pi_1(X_L)\longrightarrow\left<t_1,\dots,t_\mu\right>\,,\quad[\gamma]\longmapsto\prod_{i=1}^\mu t_i^{\lk(\gamma,L_i)}\,,
	\]
	where~$\Z^\mu$ is now denoted multiplicatively. (We could have used the~$\Z^\mu$-cover of~$M_L$ instead.)
	We start with the obvious decomposition~$H_*(p^{-1}(\partial\nu(L))=\bigoplus_{K\subset L}H_*(p^{-1}(\partial\nu(K))$. For each component~$K$ of~$L$, the torus~$\partial\nu(K)$ lifts to~$p^{-1}(\partial\nu(K))$ which consists in disjoint copies
	of cylinders if~$\varphi(\ell_K)=1$, and of planes otherwise. This implies
	\begin{equation}
	\label{eq:nu-2}
	H_2(\partial\nu(L);\Lambda)=0
	\end{equation}
and
	\begin{equation}
	\label{eq:nu-1}
	H_1(\partial\nu(L);\Lambda)=\bigoplus_{K\subset L}H_1(p^{-1}(\partial\nu(K))\simeq\bigoplus_{\substack{K\subset L\\\varphi(\ell_K)=1}}\Lambda/(\varphi(m_K)-1)\,.
	\end{equation}
	Furthermore, a straightforward computation using the definition of twisted homology yields
		\begin{equation}
		\label{eq:nu-0}
	H_0(\partial\nu(L);\Lambda)=\bigoplus_{K\subset L}H_0(\partial\nu(K);\Lambda)\simeq\bigoplus_{K\subset L}\Lambda/(\varphi(m_K)-1,\varphi(\ell_K)-1)\,.
	\end{equation}
	
	Let us go back to the exact sequence~\eqref{eq:MV}, and first assume that~$\Delta_L$ vanishes. This implies that the rank of~$H_1(X_L;\Lambda)$, i.e. the dimension
	of~$H_1(X_L;\Lambda)\otimes_\Lambda Q$, is positive, where~$Q=Q(\Lambda)$ stands for the quotient field of~$\Lambda$. The~$\Lambda$-module~$Q$ being flat, the sequence~\eqref{eq:MV} remains exact when tensored by~$Q$.
	From the equalities~\eqref{eq:nu-2}--\eqref{eq:nu-0} above, the~$Q$-vector spaces~$H_*(\partial\nu(L);\Lambda)\otimes_\Lambda Q$ vanish. Furthermore, one can check that~$H_1(P_L;\Lambda)\otimes Q\simeq H_1(P_L;Q)=0$,
	see the proof of~\cite[Lemma~3.3]{BFP16}. (This also follows from computations below.) Therefore, we
	get an isomorphism
	\begin{equation*}
	H_1(X_L;\Lambda)\otimes_\Lambda Q\simeq\A(L)\otimes_\Lambda Q\,.
	\end{equation*}
	As a consequence, the rank of~$\A(L)$ is positive, implying~$\widetilde\Delta_L=0$, and the proposition holds.
	
	From now on, we assume~$\Delta_L\neq 0$, which by the isomorphism above is equivalent to~$\widetilde\Delta_L\neq 0$.
	Given a finitely presented~$\Lambda$-module~$H$, let~$\Delta_H\in\Lambda$ denote a greatest common divisor
	of its first elementary ideal (see e.g~\cite[Chapter~6]{Lic97}). Note that by~\eqref{eq:nu-1} and~\eqref{eq:nu-0}, we have
	\begin{equation}
	\label{eq:Delta-nu}
	\Delta_{H_1(\partial\nu(L);\Lambda)}\stackrel{\cdot}{=}\prod_{\substack{K\subset L\\ \varphi(\ell_K)=1}}(\varphi(m_K)-1)\stackrel{\cdot}{=}\Delta_{H_0(\partial\nu(L);\Lambda)}\,.
	\end{equation}
	By a recursive use of~\cite[Lemma~5]{Lev67} (see also~\cite[Lemma~7.2.7]{Kau96} for a more recent proof)
	applied to~\eqref{eq:MV} together with~\eqref{eq:nu-2} and~\eqref{eq:Delta-nu},
	we have the following equality in~$\Lambda$:
	\begin{equation*}
	\Delta_{H_2(X_L;\Lambda)}\,\Delta_{H_2(P_L;\Lambda)}\,\widetilde{\Delta}_L\,\Delta_{H_0(X_L;\Lambda)}\,\Delta_{H_0(P_L;\Lambda)}\stackrel{\cdot}{=}
	\Delta_{H_2(M_L;\Lambda)}\,\Delta_L\,\Delta_{H_1(P_L;\Lambda)}\,\Delta_{H_0(M_L;\Lambda)}\,.
	\end{equation*}
	Since~$X_L$ and~$M_L$ are both connected, we have
	\[
	H_0(X_L;\Lambda)\simeq H_0(M_L;\Lambda)\simeq\Lambda/(t_1-1,\dots,t_\mu-1)\,.
	\]
	Also, one easily checks the well-known fact that~$H_2(X_L;\Lambda)=0$ if~$\Delta_L\neq 0$. Indeed, since~$X_L$ has the homotopy
	type of a~$2$-dimensional complex, the~$\Lambda$-module~$H_2(X_L;\Lambda)$ is torsion free.
	On the other hand, as~$H_0(X_L;\Lambda)\simeq\Z$ is torsion and~$H_1(X_L;\Lambda)$ is torsion since~$\Delta_L\neq 0$, we have~$0=\chi(X_L)=\dim H_2(X_L;Q)=\dim H_2(X_L;\Lambda)\otimes_\Lambda Q$ so~$H_2(X_L;\Lambda)$ is torsion. This~$\Lambda$-module being torsion free and torsion, it is trivial as claimed.
	 We now get
	\begin{equation}
		\label{eq:Delta}
	\Delta_{H_2(P_L;\Lambda)}\,\widetilde{\Delta}_L\,\Delta_{H_0(P_L;\Lambda)}\stackrel{\cdot}{=}
	\Delta_{H_2(M_L;\Lambda)}\,\Delta_L\,\Delta_{H_1(P_L;\Lambda)}\,.
	\end{equation}
	
	Since~$M_L$ is a closed oriented~3-manifold, we can apply Poincaré duality with twisted coefficients (see~\cite[Theorem~243.1]{Friedl}), yielding an isomorphism~$H_2(M_L;\Lambda)\simeq H^1(M_L;\Lambda)$.
	Furthermore, by the Universal Coefficient Spectral Sequence~\cite[Theorem~2.3]{Lev77} applied
	to the cellular chain complex of~$\Lambda$-modules~$C_*(\widetilde{M}^\varphi_L)$, we have an exact sequence
	\[
	0\longrightarrow\operatorname{Ext}_\Lambda^1(H_0(M_L;\Lambda),\Lambda)\longrightarrow
	H^1(M_L;\Lambda)\longrightarrow\operatorname{Hom}_\Lambda(H_1(M_L;\Lambda),\Lambda)\,.
	\]
	(Here, we used Propositions~239.2 and~239.6 of~\cite{Friedl} to compute the twisted homology and cohomology
	of~$M_L$ via its cover~$\widetilde{M}_L^\varphi$.)
	 The~$\Lambda$-module~$\A(L)=H_1(M_L;\Lambda)$ is torsion since~$\widetilde\Delta_L\neq 0$, so~$\operatorname{Hom}_\Lambda(H_1(M_L;\Lambda),\Lambda)$ vanishes, implying
	 \[
	 H^1(M_L;\Lambda)\simeq\operatorname{Ext}_\Lambda^1(H_0(M_L;\Lambda),\Lambda)\simeq\operatorname{Ext}_\Lambda^1(\Z,\Lambda)
	 \]
	 since~$M_L$ is connected. The definition of group cohomology and its identification
	 with cohomology of the Eilenberg-MacLane space (see e.g.~\cite[Chapter~I, Proposition~4.2]{Bro94}) implies
	 \[
	 \operatorname{Ext}_\Lambda^1(\Z,\Lambda)=\operatorname{Ext}_{\Z[\Z^\mu]}^1(\Z,\Lambda)=H^1(\Z^\mu;\Lambda)=H^1(\mathbb{T}^\mu;\Lambda)\,.
	 \]
	 Finally, Poincaré duality and the definition of twisted homology yields
	 \[
	 H^1(\mathbb{T}^\mu;\Lambda)\simeq H_{\mu-1}(\mathbb{T}^\mu;\Lambda)=H_{\mu-1}(\R^\mu)\,,
	 \]
	and we obtain
	\begin{equation}
	\label{eq:H2ML}
	H_2(M_L;\Lambda)\simeq \begin{cases}\Lambda/(t-1)&\text{if~$\mu=1$;}\\ 0&\text{else}.\end{cases}
	\end{equation}
	
	We finally come to the computation of the internal twisted homology modules~$H_*(P_L;\Lambda)$,
	i.e. of the~$\Lambda$-modules~$H_*(p^{-1}(P_L))$, where~$p\colon\widetilde{M}_L\to M_L$ denotes the~$\Z^\mu$-cover corresponding to the meridional homomorphism~$\varphi\colon\pi_1(M_L)\to\left<t_1,\dots,t_\mu\right>$. (Once again, we denote~$\Z^\mu$ multiplicatively.)
	Throughout this part of the proof, we will use the notation~$\widetilde{Y}:=p^{-1}(Y)$ for any subspace~$Y\subset M_L$.
	Recall that~$P_L$ is constructed by gluing together~3-manifolds~$D^\circ_K\times S^1$ indexed by components~$K\subset L$ along tori~$\mathbb{T}_e$ indexed by edges~$e$ of the graph~$\Gamma_L$ (see Section~\ref{sub:ML}).	
	Our main tool is the following exact sequence, which appears (in a slightly different form, with twisted~$Q$-coefficients) in the proof of~\cite[Lemma~3.3]{BFP16}, and (with untwisted real coefficients) in the proof of~\cite[Lemma~4.7]{CNT}.
Consider the exact sequence of cellular chain complexes
\[
0\longrightarrow\bigoplus_{e} C_*(\widetilde{\mathbb{T}_e})\longrightarrow\bigoplus_{K\subset L} C_*(\widetilde {D^\circ_K\times S^1})\longrightarrow C_*(\widetilde{P_L}
)\longrightarrow 0\,,
\]
where the first map sends a cell in~$\widetilde{\mathbb{T}_e}$ to the difference of the corresponding cells in~$\widetilde {D^\circ_K\times S^1}$ and in~$\widetilde {D^\circ_{K'}\times S^1}$ if~$e$ links~$K$ and~$K'$.
This induces a long exact sequence of~$\Lambda$-modules
\[
\begin{aligned}
\bigoplus_{K\subset L} H_2(\widetilde{D^\circ_K\times S^1})\longrightarrow &H_2(\widetilde{P_L})\longrightarrow\bigoplus_{e} H_1(\widetilde{\mathbb{T}_e})\longrightarrow\bigoplus_{K\subset L} H_1(\widetilde{D^\circ_K\times S^1})\\
\longrightarrow &H_1(\widetilde{P_L})
\longrightarrow\bigoplus_{e} H_0(\widetilde{\mathbb{T}_e})\longrightarrow\bigoplus_{K\subset L} H_0(\widetilde{D^\circ_K\times S^1})\longrightarrow H_0(\widetilde{P_L})\longrightarrow 0\,.
\end{aligned}
\]
Since the homomorphism~$\varphi$ maps~$[\ast_i\times S^1]$ to~$t_i$ for any~$\ast_i\in D_K$ with~$K\subset L_i$
while~$D^\circ_K$ retracts to a wedge of circles, we have~$H_2(\widetilde{D^\circ_K\times S^1})=0$ for all~$K$.
Also, since~$\varphi$ maps the meridian and longitude of~$\mathbb{T}_e$ to~$t_i$ and~$t_j$ if~$e$ links
components of colors~$i$ and~$j$, we have~$H_1(\widetilde{\mathbb{T}_e})=0$ and~$H_0(\widetilde{\mathbb{T}_e})\simeq\Lambda/(t_i-1,t_j-1)$ for all~$e$. As a consequence, we obtain
\begin{equation}
\label{eq:H2P}
H_2(P_L;\Lambda)=H_2(\widetilde{P_L})=0\,.
\end{equation}
Applying once again Levine's~\cite[Lemma~5]{Lev67} to the resulting exact sequence of~$\Lambda$-modules
\[
0\to\bigoplus_{K\subset L} H_1(\widetilde{D^\circ_K\times S^1})
\to H_1(\widetilde{P_L})
\to\bigoplus_{e} H_0(\widetilde{\mathbb{T}_e})\to\bigoplus_{K\subset L} H_0(\widetilde{D^\circ_K\times S^1})\to H_0(\widetilde{P_L})\to 0\,,
\]
and using the fact that~$\Delta_{H_0(\widetilde{\mathbb{T}_e})}\stackrel{\cdot}{=}1$ since edges of~$\Gamma_L$ link components of different colors, we get
\begin{equation}
\label{eq:Delta'}
\Delta_{H_0(P_L;\Lambda)}\prod_{K\subset L}\Delta_{H_1(\widetilde{D^\circ_K\times S^1})}\stackrel{\cdot}{=}\Delta_{H_1(P_L;\Lambda)}\prod_{K\subset L}\Delta_{H_0(\widetilde{D^\circ_K\times S^1})}
\end{equation}
Equations~\eqref{eq:Delta},~\eqref{eq:H2P} and~\eqref{eq:Delta'} now yield
\begin{equation}
\label{eq:Delta''}
\widetilde{\Delta}_L\,\prod_{K\subset L}\Delta_{H_0(\widetilde{D^\circ_K\times S^1})}\stackrel{\cdot}{=}
	\Delta_{H_2(M_L;\Lambda)}\,\Delta_L\,\prod_{K\subset L}\Delta_{H_1(\widetilde{D^\circ_K\times S^1})}\,.
\end{equation}
Note that~$D^\circ_K$ deformation retracts onto a wedge of~$\sum_{K'\subset L}\vert\lk(K,K')\vert$ circles,
the sum being over all~$K'$ of color different from the color of~$K$. It is not very difficult to compute the corresponding
twisted homology modules. Alternatively, and following~\cite{BFP16}, one can make a small detour through Reidemeister
torsion and use the existing literature on the topic: by~\cite[Theorem~4.7]{Tur01} and~\cite[Example~2.7]{Nic03}, we have
\begin{equation}
\label{eq:RT}
\frac{\Delta_{H_1(\widetilde{D^\circ_K\times S^1})}}{\Delta_{H_0(\widetilde{D^\circ_K\times S^1})}}\de (\varphi(m_K)-1)^{-\chi(D^\circ_K)}\,,
\end{equation}
an equality in the field of fractions~$Q$ up to multiplication by units of~$\Lambda$.
(Note that this tool could also have been used to obtain the equality~\eqref{eq:Delta-nu}.) In the case~$\mu=1$, Equations~\eqref{eq:H2ML},~\eqref{eq:Delta''} and~\eqref{eq:RT} imply
\[
\widetilde{\Delta}_L=(t-1)\prod_{K\subset L}(\varphi(m_K)-1)^{-1}\Delta_L=\frac{\Delta_L}{(t-1)^{\vert L\vert -1}}
\]
while in the case~$\mu>1$, they yield
\[
\widetilde{\Delta}_L=\prod_{K\subset L}(\varphi(m_K)-1)^{-\chi(D_K^\circ)}\,\Delta_L=\prod_{i=1}^\mu(t_i-1)^{\nu_i}\,\Delta_L
\]
with
\[
\nu_i=\sum_{K\subset L_i}-\chi(D_K^\circ)=\sum_{K\subset L_i}\left(\Big(\sum_{K'\subset L\setminus L_i}\vert\lk(K,K')\vert\Big)-1\right)=\Big(\sum_{\substack{K\subset L_i\\ K'\subset L\setminus L_i}}\vert\lk(K,K')\vert\Big)-\vert L_i\vert\,.
\]
This concludes the proof.
	\end{proof}

	We end this section with a number of remarks: the first one pertains to the Alexander modules,
	the remaining four to the Alexander polynomials.
	
	\begin{rems}
	\label{rems:prop}
	\begin{enumerate}
	\item\label{rems:prop-1} In the~$\mu=1$ case with $\Delta_L \neq 0$, the Mayer-Vietoris exact sequence~\eqref{eq:MV} takes the simple form
	\[
	0\longrightarrow\Lambda/(t-1)\longrightarrow\bigoplus_{K\subset L}\left(\Lambda/(t-1)\right)[\ell_K]\longrightarrow H_1(X_L;\Lambda)\longrightarrow\A(L)\longrightarrow 0\,,
	\]
	with the first homomorphism mapping a generator to~$\sum_{K\subset L}[\ell_K]$.
	Thus, we have an isomorphism~$\A(L)\simeq H_1(X_L;\Lambda)/\left<[\ell_K]\mid K\subset L\right>$
	which boils down to the aforementioned isomorphism~$\A(L)\simeq H_1(X_L;\Lambda)$ in the case of knots (recall Remark~\ref{rems:Hoso}\ref{rems:Hoso-1}).
	However, in the general case of~$\mu>1$, there does not seem to be any straightforward relation between the modules~$\A(L)$ and~$H_1(X_L;\Lambda)$.
		
		\item Since the (multivariable) Alexander polynomial admits a natural normalisation in the form of the Conway function~$\nabla_L$~\cite{Con67,Har}, it is possible to give a well-defined representative of~$\widetilde\Delta_L$
	in the ring~$\Z[t_1^{\pm 1/2},\dots,t_\mu^{\pm 1/2}]$ via
	\[
	\widetilde\Delta_L(t):=\frac{\nabla_L(t^{1/2})}{(t^{1/2}-t^{-1/2})^{|L|-2}}
	\]
	if~$\mu=1$, and
	\[
	\widetilde\Delta_L(t_1,\dots,t_\mu):=\prod_{i=1}^\mu (t_i^{1/2}-t_i^{-1/2})^{\nu_i}\,\nabla_L(t_1^{1/2},\dots,t_\mu^{1/2})
	\]
	if~$\mu>1$. Most probably, this renormalized version of~$\widetilde\Delta_L$ can be constructed using
	Turaev's sign-refined Reidemeister torsion~\cite{Tur86} of the manifold~$M_L$, but we shall not attempt to do so.

	\item In the case~$\mu=1$, the renormalized Alexander polynomial
	\[
	\widetilde\Delta_L(t)\stackrel{\cdot}{=}\frac{\Delta_L(t)}{(t-1)^{|L|-1}}\de\frac{\Delta_L(t,\dots,t)}{(t-1)^{|L|-2}}
	\]
	is nothing but the {\em Hosokawa polynomial\/} defined in~\cite{Hos}, hence the terminology. (Hosokawa
	used the notation~$\nabla$, which we avoid to prevent confusion with the Conway function.)

	\item In the case~$\mu=1$, Hosokawa showed that~$\widetilde{\Delta}_L\in\Z[t,t^{-1}]$ is a symmetric polynomial of even degree with~$\pm\widetilde{\Delta}_L(1)$ equal to the determinant of the reduced linking matrix of~$L$.
	The symmetry property obviously carries over to the case~$\mu>1$. Also, using the Universal Coefficient Spectral Sequence, it is possible to give a homological interpretation of~$\widetilde{\Delta}_L(1,\dots,1)$ in terms of the
	renormalized Alexander module~$\A(L)$. However, it does not seem to admit a simple formulation.
	In particular, it is not determined by the linking numbers, since any symmetric polynomial in~$\Z[t_1^{\pm 1/2},t_2^{\pm 1/2}]$ can be realised as the~2-variable Hosokawa polynomial of a~2-component link with vanishing linking number (see e.g.~\cite{Pla86}). 
	
	\item In the case~$\mu=1$, Hosokawa also showed that any symmetric polynomial can be realised as the Hosokawa polynomial of an oriented link of any given number of components.
	In general, the algebraic characterisation of the multivariable Hosokawa polynomial
	boils down to the corresponding question for the multivariable Alexander polynomial, aka Problem~2 in
	Fox's list~\cite{Fox61}, a notoriously difficult question (see e.g.~\cite{Pla86} and references therein).
	\end{enumerate}
	\end{rems}

	\section{Piecewise continuity of the signature}
	\label{sec:cont}
	
The aim of this section is to show that the behavior of the extended signature and nullity of a colored link~$L$ is strongly related to the renormalized Alexander module~$\A(L)$. In particular, we will see that the extended signature
is constant on the connected components of the complement of the zeros of the Hosokawa polynomial.

\medskip
	
	Let~$\mathcal{E}_r(L)\subset\Lambda$ be the~$r^\text{th}$-elementary ideal of the renormalized Alexander
	module~$\A(L)$, i.e. the ideal generated by the~$(m-r+1)\times(m-r+1)$-minors of a presentation matrix of~$\A(L)$ with~$m$
	generators and~$n$ relations, where we assume~$n\ge m$. By convention, we set~$\mathcal{E}_r(L)=0$ for~$r\le 0$ and~$\mathcal{E}_r(L)=\Lambda$ for~$r> m$. 

	\begin{thm}
		\label{thm:main}
		The algebraic subsets
		\[
		\Sigma_r(L):=\{\omega\in\T^\mu\setminus\{(1,\dots,1)\}\mid p(\omega)=0\text{ for all } p\in\mathcal{E}_{r}(L)\}
		\]
		form a filtration or the pointed torus
		\[
		\T^\mu\setminus\{(1,\dots,1)\}=\Sigma_0(L)\supset\Sigma_1(L)\supset\dots\supset\Sigma_\ell(L)=\emptyset
		\]
		such that for all~$r$, the signature~$\sigma_L$ is constant on the connected components of~$\Sigma_r(L)\setminus\Sigma_{r+1}(L)$. Moreover, the nullity~$\eta_L(\omega)$ is equal to~$r$ for~$\omega\in\Sigma_r(L)\setminus\Sigma_{r+1}(L)$ as long as at most two coordinates of~$\omega$ are equal to~$1$.
		Finally, if~$\mu=1$, then these facts hold on the full circle~$\T^1=S^1$.
	\end{thm}

	Of course, an analogous result is known for non-extended signatures, see~\cite[Theorem~4.1]{CF08}. Let us also emphasize that the proof of Theorem~\ref{thm:main} heavily relies on the homological computations from~\cite[Appendix B]{CMP23}.

	\begin{proof}
	Let~$L$ be an arbitrary~$\mu$-colored link. As in Section~\ref{sub:ext}, let~$\varphi\colon H_1(M_L)\to\Z^\mu$ be a meridional homomorphism on the generalized Seifert surgery on~$L$, thus defining integers~$\mu_L\in\Z^{\mu\choose 3}$ and the auxiliary link~$L^\#$ via~\eqref{eq:Lsharp}. The map~$\varphi$ extends
	to~$\varphi^\#\colon H_1(M_{L^\#})\to\Z^\mu$, and there exists a compact oriented $4$-manifold~$W$ endowed
	with an isomorphism~$\Phi\colon\pi_1(W)\stackrel{\simeq}{\to}\Z^\mu$ such that~$\partial W= M_{L^\#}$ and~$\Phi$ extends~$\varphi^\#$. By definition, we have~$\sigma_L(\omega)=\sigma_\omega(W)$ while
	\[
	\eta_L(\omega)=\eta_\omega(W)-\sum_{i<j<k}\vert\mu_L(ijk)\vert-2\sum_{\stackrel{i<j<k}{\omega_i=\omega_j=\omega_k=1}}\vert\mu_L(ijk)\vert\,.
	\]
	Note in particular that as long as at most two coordinates of~$\omega$ are equal to~$1$, we have the equality~$\eta_L(\omega)=\eta_\omega(W)-\vert\mu_L\vert$, where~$\vert\mu_L\vert$ stands for~$\sum_{i<j<k}\vert\mu_L(ijk)\vert$.
	
	\medskip
	
		In order to prove the statement, it will be convenient to work with coefficients in the extended ring~$\LambdaC\coloneq\Lambda\otimes_\mathbb{Z}\mathbb{C}=\mathbb{C}[t_1^{\pm 1},\cdots,t_{\mu}^{\pm 1}]$. Note that this change is irrelevant to the result, since we are only interested in vanishing sets of polynomials. Moreover, for any~$j\in\{1,\cdots,\mu\}$, we set~$\LambdaC_j\coloneq\LambdaC[(t_j-1)^{-1}]$ to be the localized ring obtained by adjoining the inverse of~$t_j-1$ to~$\LambdaC$.
	
		First note that since~$W$ satisfies~$\pi_1(W) \simeq \mathbb{Z}^{\mu}$, we have~$H_1(W;\LambdaC) = 0$. Furthermore, from the proof of~\cite[Lemma B.1]{CMP23}, we know that the intersection form on~$H_2(W;\LambdaC_j)$ can be represented by a matrix which can in turn be used for computing the intersection form on~$H_2(W;\C^\omega)$. The main goal of the proof will thus be to relate a presentation matrix for~$\A(L)$ with a matrix representing the intersection form on~$H_2(W;\LambdaC_j)$. For this purpose, we need the following fact, whose proof is a standard application of the Universal Coefficient Spectral Sequence (in the following, abbreviated UCSS)~\cite[Theorem 2.3]{Lev77}: we claim that~$H_2(W,M_{L^\#}; \LambdaC_j)$ is a free~$\LambdaC_j$-module for all~$j\in\{1,\cdots,\mu\}$.
	
		Indeed, we have the Poincaré-Lefschetz duality isomorphism~$H_2(W,M_{L^\#};\LambdaC_j)\cong H^2(W; \LambdaC_j)$. By the UCSS applied to the (cellular) chain complex of the universal cover of~$W$, we have a spectral sequence
		$$E_2^{p,q}=\operatorname{Ext}_{\LambdaC}^q(H_p(W;\LambdaC),\LambdaC_j) \implies H^{p+q} (W;\LambdaC_j)$$
		with differentials of bidegree~$(1-r,r)$.
		Morerover, since~$\LambdaC_j$ is a flat~$\LambdaC$-module (being constructed by localization), by change of rings for~$\operatorname{Ext}$ (see e.g. \cite[Chapter~IV, Proposition~12.2]{HS97}), we have
		$$ \operatorname{Ext}_{\LambdaC}^q(H_p(W;\LambdaC),\LambdaC_j) \cong \operatorname{Ext}_{\LambdaC_j}^q(H_p(W;\LambdaC)\otimes_{\LambdaC} \LambdaC_j,\LambdaC_j) \cong \operatorname{Ext}_{\LambdaC_j}^q(H_p(W;\LambdaC_j),\LambdaC_j)\,, $$
		where in the second isomorphism we use the fact that~$\LambdaC_j$ is flat over~$\LambdaC$.	
		Now, by~\cite[Lemma~B.8]{CMP23}, the module~$H_p(W;\LambdaC_j)$ vanishes if~$p\neq 2$, so~$E_2^{p,q} = 0$ if~$p\neq 2$. For all~$n$, the UCSS then yields the isomorphism~$H^n(W;\LambdaC_j) \cong \operatorname{Ext}_{\LambdaC_j}^{n-2}(H_2(W;\LambdaC_j), \LambdaC_j)$, and in particular
		$$ H^2(W;\LambdaC_j) \cong \operatorname{Ext}_{\LambdaC_j}^0(H_2(W;\LambdaC_j),\LambdaC_j) \cong \operatorname{Hom}_{\LambdaC_j}(H_2(W;\LambdaC_j), \LambdaC_j)\,.$$
		Since~$H_2(W;\LambdaC_j)$ is a free~$\LambdaC_j$-module by~\cite[Lemma B.7]{CMP23}, we can conclude that the~$\LambdaC_j$-module $H_2(W,M_{L^\#};\LambdaC_j) \cong H^2 (W;\LambdaC_j)$ is free as well, which proves the claim.
	
		Therefore, since~$H_2(W;\LambdaC_j)$ and~$H_2(W,M_{L^\#};\LambdaC_j)$ are both free of the same rank
		and~$H_1(W;\LambdaC_j)$ vanishes, the long exact sequence of the pair gives a presentation
		$$H_2(W; \LambdaC_j)\stackrel{i_*}{\longrightarrow} H_2(W,M_{L^\#}; \LambdaC_j) \longrightarrow H_1(M_{L^\#}; \LambdaC_j) \longrightarrow 0$$
		of~$H_1(M_{L^\#};\LambdaC_j)$. Let~$A_j$ be a (square) matrix representing~$i_*$ in the above presentation. Recall also that the intersection form on~$H_2(W;\LambdaC_j)$ is defined by 
		$$
		H_2(W;\LambdaC_j) \stackrel{i_*}{\longrightarrow} H_2(W,M_{L^\#};\LambdaC_j) \stackrel{\mathrm{PD}}{\longrightarrow} H^2(W;\LambdaC_j)\stackrel{\mathrm{ev}}{\longrightarrow} \operatorname{Hom}_{\LambdaC_j}(H_2(W;\LambdaC_j),\LambdaC_j)\,,$$
		where second map is the Poincaré duality isomorphism and the third is the evaluation isomorphism arising from the UCSS, as explained above. Let~$H_j$ be a matrix representing the intersection form on~$H_2(W;\LambdaC_j)$. By naturality of the intersection form \cite[Lemma B.5]{CMP23}, for any~$\omega\in\T^\mu$ with~$\omega_j \neq 1$, the intersection form on~$H_2(W;\C^\omega)$ is represented by~$H_j(\omega)$, and similarly, the map~$i_*:H_2(W;\C^\omega)\longrightarrow H_2(W,M_{L^\#};\C^\omega)$ is represented by~$A_j(\omega)$.
	
		To conclude, take~$\omega\in\T^\mu \setminus\{1,\dots,1\}$ and choose a~$j\in\{1,\dots,\mu\}$ such that~$\omega_j\neq 1$. Let~$n$ denote~$\operatorname{rank}_\mathbb{C}H_2(W;\C^\omega)=\operatorname{rank}_{\LambdaC_j}H_2(W;\LambdaC_j)$: once again, the equality follows from the UCSS, as explained in the proof of \cite[Lemma B.1]{CMP23}. Then, for any~$r\geq 0$,
		$$
		\begin{aligned}
			\eta_\omega(W) = \nul(H_j(\omega)) \geq r &\iff \operatorname{rank}H_j(\omega) \leq n-r \iff \operatorname{rank}A_j(\omega) \leq n-r\\
			&\iff \text{all } (n-r+1)\text{-minors of } A_j(\omega) \text{ vanish}\\
			&\iff p(\omega) = 0 \text{ for all }p\in \mathcal{E}_r(A_j).\\
		\end{aligned}$$
		Since~$A_j$ is a presentation matrix of~$H_1(M_{L^\#};\LambdaC_j)\cong \A({L^\#})\otimes_{\Lambda}\LambdaC_j$, it follows that 
		\begin{equation}
		\label{eq:eta}
		\begin{aligned}
		\eta_\omega(W)\geq r &\iff p(\omega) = 0 \text{ for all }p\in \mathcal{E}_r(\A({L^\#})\otimes_{\Lambda}\LambdaC_j)\\
		&\iff p(\omega) = 0 \text{ for all }p\in \mathcal{E}_r({L^\#})
		\iff\omega\in\Sigma_r({L^\#})\,.
		\end{aligned}
		\end{equation}
		Thus, the nullity~$\eta_\omega(W)=\nul(H_j(\omega))$ is constant equal to~$r$ on~$\Sigma_r({L^\#})\setminus\Sigma_{r+1}({L^\#})$. Since the
		signature~$\sigma_L(\omega)=\sign(H_j(\omega))$ can only change value when the
		nullity~$\eta_\omega(W)=\nul(H_j(\omega))$ changes value, we deduce that~$\sigma_L$ is constant on the connected components of~$\Sigma_r({L^\#})\setminus\Sigma_{r+1}({L^\#})$.
		
		To conclude the proof, we now show the equality
	\begin{equation}
	\label{eq:Sigma}
	\Sigma_r(L^\#)=\Sigma_{r-\vert\mu_L\vert}(L)\subset\T^\mu\setminus\{(1,\dots,1)\}
	\end{equation}
	for all~$r\ge 0$.
	To do so, we can assume~$\mu\ge 3$, as~$L^\#=L$ and~$\vert\mu_L\vert=0$ otherwise. Clearly, it is enough check that if~$L^\#$ is given by the distant sum of~$L$ with one copy of a $3$-colored (and arbitrarily oriented) Borromean ring~$B$,
	then we have the equality~$\Sigma_r(L^\#)=\Sigma_{r-1}(L)$ in~$\T^\mu\setminus\{(1,\dots,1)\}$, as the general case
	can be recovered inductively. By construction,
	the manifold~$M_{L^\#}$ is then given by the connected sum of~$M_L$ with~$M_{B}$, which by Example~\ref{exs:ML}\ref{exs:ML-1} is the $3$-torus endowed with the homomorphism~$H_1(S^1\times S^1\times S^1)\to\Z^\mu$
	determined by the orientations and colors of the components of~$B$. A straighforward Mayer-Vietoris argument (using Theorem~240.7 of~\cite{Friedl} as in the proof of Proposition~\ref{prop:Hosokawa}) yields the exact sequence
	\[
	0\longrightarrow\A(L)\longrightarrow\A(L^\#)\longrightarrow\Lambda\stackrel{\varepsilon}{\longrightarrow}\Z\longrightarrow 0\,,
	\]
	where~$\varepsilon$ is the augmentation homomorphism defined by~$\varepsilon(p)=p(1,\dots,1)$.
	By~\cite{Tra82}, we have the inclusions
	\[
	\mathcal{E}_{r-1}(L)\cdot \mathcal{K}^{\mu-1}\subset \mathcal{E}_r(L^\#)\quad\text{and}\quad\mathcal{E}_r(L^\#)\cdot \mathcal{K}^{{\mu-1}\choose{2}}\subset\mathcal{E}_{r-1}(L)\,,
	\]
	where~$\mathcal{K}$ stands for the augmentation ideal~$\Ker(\varepsilon)$. Now, assume that~$\omega\in\T^\mu\setminus\{(1,\dots,1)\}$ belongs to~$\Sigma_r(L^\#)$. By definition, we have~$p^\#(\omega)=0$ for
	all~$p^\#\in\mathcal{E}_r(L^\#)$; by the first inclusion displayed above, we have in particular~$p(\omega)q(\omega)^{\mu-1}=0$ for all~$p\in\mathcal{E}_{r-1}(L)$ and all~$q\in\mathcal{K}$.
	Fixing an arbitrary~$p\in\mathcal{E}_{r-1}(L)$, we either have~$p(\omega)=0$, or~$q(\omega)=0$ for all~$q\in\K$.
	The latter is excluded, since any~$\omega\neq (1,\dots,1)$ admits a polynomial~$q=t_i-1\in\K$ such that~$q(\omega)\neq 0$. As we explicitely excluded the value~$\omega=(1,\dots,1)$, we have~$p(\omega)=0$. Therefore, the element~$\omega$ belongs to~$\Sigma_{r-1}(L)$, and the inclusion~$\Sigma_r(L^\#)\subset\Sigma_{r-1}(L)$ is checked. The reverse inclusion can be verified in the same way using the second inclusion diplayed above.
	
		Recall that we have the equality~$\eta_L(\omega)=\eta_\omega(W)-\vert\mu_L\vert$ as long as at most two coordinates of~$\omega$ are equal to~$1$. Therefore, by~\eqref{eq:eta} and~\eqref{eq:Sigma}, the nullity~$\eta_L(\omega)$ is constant equal to~$r$ on~$\Sigma_r\setminus\Sigma_{r+1}$ as long as this condition is satisfied.
	
		Let us finally assume that~$\mu = 1$. Since~$\LambdaC = \C[t_1^{\pm 1}]$ is a PID, the Universal Coefficient Theorem directly implies that~$H_2(W;\LambdaC)$ is free. All the arguments above go through directly without needing to work in the localized rings~$\LambdaC_j$ (compare with the proof of~\cite[Lemma B.1]{CMP23}). Therefore, the result holds on the full circle~$\T^1$.
	\end{proof} 
	
	Theorem~\ref{thm:main} and Proposition~\ref{prop:Hosokawa} directly imply the following corollaries,
	which deal with the $\mu=1$ and~$\mu>1$ cases, respectively.
	
	\begin{cor}
	\label{cor:cont1}
	The extended Levine-Tristram signature~$\sigma_L\colon S^1\to\Z$ of an oriented link~$L$ is constant on
	the connected components of the complement of the zeros of the Hosokawa polynomial~$\widetilde\Delta_L(t)=\frac{\Delta_L(t)}{(t-1)^{|L|-1}}$.\qed
	\end{cor}
	
	In particular, if~$(t-1)^{|L|}$ does not divide~$\Delta_L$ in~$\Lambda$, then~$\lim_{\omega\to 1}\sigma_L(\omega)$ is equal to~$\sigma_L(1)$, the signature of the linking matrix of~$L$.
	This is the main result of~\cite{BZ22} (see also~\cite{CMP23}).

	\begin{exs}
	\begin{enumerate}
	\item If~$L=K$ is a knot, then~$\widetilde\Delta_K\de\Delta_K$ satisfies~$\Delta_K(1)=\pm 1$, leading to the well-known and easy fact that~$\lim_{\omega\to 1}\sigma_K(\omega)=\sigma_K(1)=0$.
	\item The (1-colored, positive) Hopf link has Alexander polynomial~$\Delta_L\de t-1$, hence Hosokawa
	polynomial~$\widetilde\Delta_L\de 1$ and constant signature~$\sigma_L\colon S^1\to\Z$ (equal to~$-1$).
	\end{enumerate}
	\end{exs}
	
	\begin{cor}
	\label{cor:cont2}
	The extended signature~$\sigma_L\colon \T^\mu\setminus\{(1,\dots,1)\}\to\Z$ of a~$\mu$-colored link~$L$ with~$\mu>1$ 
	is constant on the connected components of the complement of the zeros of the Hosokawa polynomial
	$\widetilde\Delta_L(t_1,\dots,t_\mu)\stackrel{\cdot}{=}\prod_{i=1}^\mu (t_i-1)^{\nu_i}\Delta_L(t_1,\dots,t_\mu)$, where $\nu_i=\Big(\sum_{\substack{K\subset L_i\\ K'\subset L\setminus L_i}}\vert\lk(K,K')\vert\Big)-\vert L_i\vert$.\qed
	\end{cor}
	
	Note the simple form of this polynomial in case of a~2-colored~2-component link~$L=K_1\cup K_2$, namely
	\[
	\widetilde\Delta_L(t_1,t_2)\stackrel{\cdot}{=}\left[(t_1-1)(t_2-1)\right]^{\vert\lk(K_1,K_2)\vert-1}\Delta_L(t_1,t_2)\,.
	\]
	
	\medskip
	
	Let us illustrate this corollary with several concrete examples of links, all drawn in Figure~\ref{fig:ex}.

	\begin{exs}
	\begin{enumerate}
	\item The (positive or negative)~2-colored Hopf link has linking number~$\pm 1$ and Alexander polynomial~$1$,
	hence Hosokawa polynomial~$1$ as well. By Theorem~\ref{thm:main}, its signature is
	constant (equal to~0) on~$\T^2\setminus\{(1,1)\}$. This is coherent with the easy computations from~\cite{CF08}
	and~\cite[Theorem~4.6]{CMP23}.
	
	\begin{figure}[tbp]
    \centering
   \includegraphics[width=3cm]{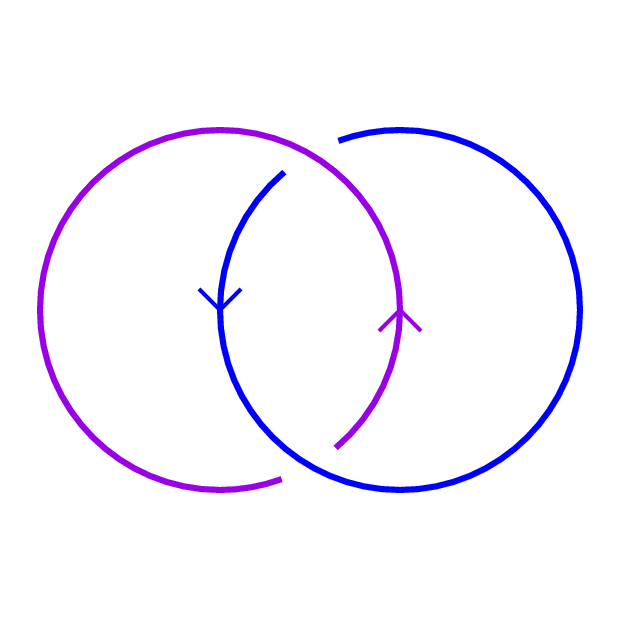}\qquad\includegraphics[width=3cm]{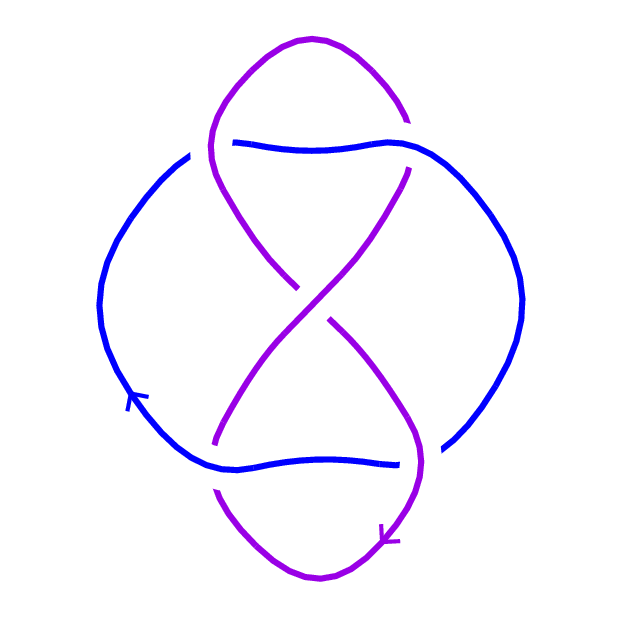} 
   \qquad\includegraphics[width=3cm]{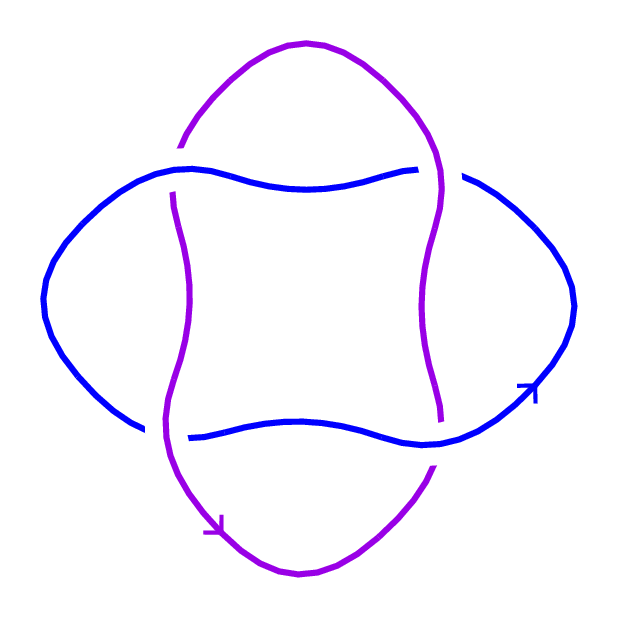}  \qquad\includegraphics[width=3cm]{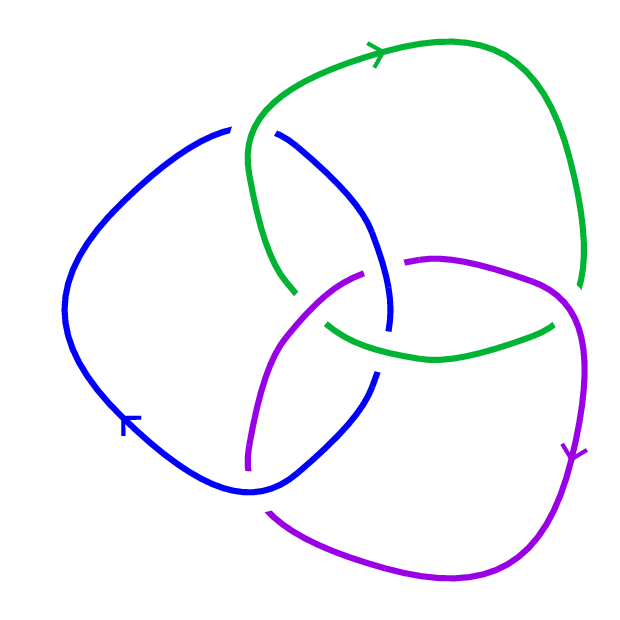}
\caption{The Hopf link, the Whitehead link, the torus link~$T(2,4)$, and the Borromean rings. (Images from LinkInfo~\cite{LinkInfo}.)}
    \label{fig:ex}
\end{figure}

	\item The Whitehead link has linking number~0 and Alexander polynomial~$(t_1-1)(t_2-1)$, hence Hosokawa polynomial~$1$. By Theorem~\ref{thm:main}, its signature is
	constant (equal to~1) on~$\T^2\setminus\{(1,1)\}$.
	 This explains the results observed in~\cite[Example~4.9]{CMP23}.
	Note however that~$\sigma_L(1,1)=0$ (as~$\sigma_L(1,1)$ only depends on the linking numbers and can easily be seen to vanish for the trivial link). This shows that in general, Theorem~\ref{thm:main} does not hold on the full torus~$\T^\mu$ for~$\mu>1$.
	
	\item The torus link~$T(2,4)$ illustrated in Figure~\ref{fig:ex} (right) has linking number~$2$ and Alexander polynomial~$t_1t_2+1$, so its signature is constant on the connected components of the complement
	of the zeros of its Hosokawa polynomial~$(t_1-1)(t_2-1)(t_1t_2+1)$.
	And indeed, one can show that the extended signature is equal to
	\[
	\sigma_L(\omega_1,\omega_2)=-\sgn[\mathrm{Re}((\omega_1-1)(\omega_2-1))]
	\]
	 for all~$(\omega_1,\omega_2)\in\mathbb{T}^2\setminus\{(1,1)\}$.
	 (This follows from~\cite[Example~2.4]{CF08} for~$\omega_1,\omega_2\in S^1\setminus\{1\}$
	and from~\cite[Theorem~4.6]{CMP23} for~$\omega_1=1$ or~$\omega_2=1$.) The graph of~$\sigma_L$ is illustrated below, where~$\mathbb{T}^2$ is pictures as a square, and the thick black lines correspond to the value~$0$.
	
	\begin{figure}[h]
\centering
\begin{overpic}[width=2cm]{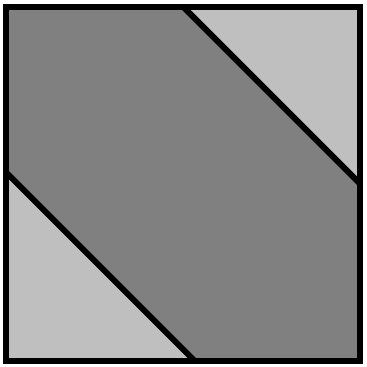}
 \put (15,15){$1$}
    \put (80,80){$1$}
        \put (40,45){$-1$}
    \end{overpic}
\end{figure}
	
	\item The~(3-colored) Borromean rings have vanishing linking numbers and Alexander polynomial $(t_1-1)(t_2-1)(t_3-1)$, hence Hosokawa polynomial~$1$ once again. Therefore, their signature is constant (equal to~0) on~$\T^3\setminus\{(1,1,1)\}$. This is consistent with the fact that the Borromean rings are amphicheiral.

	\end{enumerate}
	\end{exs}

	\begin{rem}
	Note that in general, the fact that~$\eta_L$ is equal to~$r$ on~$\Sigma_r(L)\setminus\Sigma_{r+1}(L)$ does {\em not\/} hold on the full pointed torus~$\T^\mu\setminus\{(1,\dots,1)\}$. For example, consider the $4$-colored link~$L$ given by the distant sum of a $3$-colored Borromean rings~$B$ with a trivial knot~$U$ of color $4$. Then, the associated manifold~$M_L$ is the connected sum of~$M_B=S^1\times S^1\times S^1$ and~$M_U=S^1\times S^2$.  Since both yield trivial modules~$H_1(M_B;\Lambda)=H_1(M_U;\Lambda)=0$, a Mayer-Vietoris argument leads to~$\A(L)\simeq \Ker(\varepsilon)$, the augmentation ideal. As in the proof of~\eqref{eq:Sigma}, this yields~$\Sigma_1(L)=\T^4\setminus\{(1,1,1,1)\}$ and~$\Sigma_r(L)=\emptyset$ for~$r>1$.
	By Theorem~\ref{thm:main}, we have~$\eta_L(\omega)=1$ as long as at most two coordinates of~$\omega$ are equal to~$1$. However, if~$\omega=(1,1,1,\omega_4)$ with~$\omega_4\neq 1$, then~\cite[Proposition~4.6]{CMP23} and a Mayer-Vietoris argument yield~$\eta_L(\omega)=\dim H_1(M_L;\C^\omega)=\dim H_1(M_B \# M_U;\C^\omega)=3$.	
	\end{rem}

\section{Invariance by concordance}
\label{sec:concordance}

The aim of this section is to show that the extended signature and nullity are invariant under concordance of colored links,
a result known to hold for the non-extended versions~\cite[Corollary~3.13]{CNT} (see also~\cite[Theorem~7.1]{CF08}).
We first recall the relevant definition.	

\begin{definition}
	\label{def:concordance}
	Two~$\mu$-colored links~$L=L_1\cup\dots\cup L_\mu$ and~$L'=L'_1\cup\dots\cup L'_\mu$ are said to be {\em concordant\/} if there exists a collection of embedded locally flat disjoint cylinders in~$S^3\times [0,1]$ such that each cylinder has one boundary component
	in~$L_i\subset S^3\times\{0\}$ and the other in~$L_i'\subset S^3\times \{1\}$ for some~$i$.
\end{definition}

We can now state the main result of this section.

\begin{thm}
	\label{thm:conc}
	If~$L$ and~$L'$ are concordant colored links, then~$\sigma_L(\omega)=\sigma_{L'}(\omega)$ and~$\eta_L(\omega)=\eta_{L'}(\omega)$ for all~$\omega\in\T^\mu_!:=\{\omega\in\T^\mu\mid p(\omega)\neq 0 \text{ for all $p\in\Lambda$ such that $p(1,\dots,1)=\pm 1$}\}$.
\end{thm}

Note that this definition of~$\T^\mu_!$ is a straighforward extension from~$(S^1\setminus\{1\})^\mu$ to the full torus~$\T^\mu$
of~\cite[Definition~2.14]{CNT}, itself generalizing~\cite{NP17} which corresponds to the case~$\mu=1$.
Note also that~$\T^\mu_!$ is dense in~$\T^\mu$ as it contains the elements whose
coordinates are all~$p^n$-roots of unity for some integer~$n$ and a common but arbitrary prime~$p$~\cite[Proposition~2.17]{CNT}.

Let us point out that the authors of~\cite{CNT} prove a stronger statement for the non-extended signature and nullity, namely their invariance under so-called {\em 0.5-solvable cobordisms\/}.
Here, we show a weaker statement (invariance under concordance) for the extended versions, relying on some technical lemma of~\cite{CNT} (see also~\cite{NP17,COT03}).
Therefore, we wish to reiterate the remark stated after Proposition~\ref{prop:Hosokawa}: even though the result is new and,
as far as we can tell, not obvious, we do not claim much originality for its proof.

\begin{proof}[Proof of Theorem~\ref{thm:conc}]
	Let us start with the nullity. For any~$\mu$-colored link~$L$, Proposition~4.5 of~\cite{CMP23} states that
	\[
	\eta_L(\omega)=\begin{cases}
		\dim H_1(M_L;\C^\omega)&\text{if $\omega\neq (1,\dots,1)$;}\\
		\dim H_1(M_L;\C)-\mu&\text{if $\omega= (1,\dots,1)$.}
	\end{cases}
	\]
	Now, consider the Mayer-Vietoris exact sequence for~$M_L=X_L\cup P_L$; it reads
	\[
	\begin{aligned}
		H_1(\partial\nu(L);\C^\omega)&\stackrel{{i_1}\choose{j_1}}{\longrightarrow} H_1(X_L;\C^\omega)\oplus H_1(P_L;\C^\omega)\longrightarrow H_1(M_L;\C^\omega)\\
		&\longrightarrow H_0(\partial\nu(L);\C^\omega)\stackrel{{i_0}\choose{j_0}}{\longrightarrow} H_0(X_L;\C^\omega)\oplus H_0(P_L;\C^\omega)\,.
	\end{aligned}
	\]
	Note that for any fixed~$\omega\in\T^\mu$, the vector spaces~$H_*(\partial\nu(L);\C^\omega)$
	and~$H_0(X_L;\C^\omega)$ only depend on the linking numbers and colors of the components of~$L$.
	The manifold~$P_L$ is also determined by this data.
	A priori, the vector spaces~$H_*(P_L;\C^\omega)$ could depend on the choice of the homomorphism~$H_1(P_L)\to\Z^\mu$ used to define the twisted coefficients, but one can easily check that this is not the case, see e.g. the proof of~\cite[Lemma~A.2]{CMP23}.
	Observe that the inclusion induced maps~$i_0,j_0$ are also determined by this data. Finally, one can check that~$j_1$ also only depends on the linking numbers and colors of the components of~$L$, see the proof of~\cite[Lemma~A.2]{CMP23}.
	Since linking numbers and colors are invariant under concordance, the invariance of the nullity now boils down to the invariance under concordance of the inclusion induced map~$i_1\colon H_1(\partial\nu(L);\C^\omega)\to H_1(X_L;\C^\omega)$.
	
	To make this statement more precise, fix two concordant colored links~$L$ and~$L'$. By definition, there exists a collection~$C\subset S^3\times [0,1]$ of locally flat disjoint cylinders such that each cylinder has one boundary component in~$L_i\subset S^3\times\{0\}$ and the other in~$L_i'\subset S^3\times \{1\}$ for some~$i$. Since~$C$ is locally flat, we can consider a tubular neighborhood~$\nu(C)$ of~$C$ whose complement $X_C:=(S^3\times [0,1])\setminus\nu(C)$ restricts to~$X_C\cap(S^3\times\{0\})=X_L$ and~$X_C\cap(S^3\times\{1\})=X_{L'}$.
	This concordance defines a homeomorphism~$\phi\colon\partial\nu(L)\to\partial\nu(L')$ which maps a meridian
	(resp. Seifert longitude) of any component of~$L$ to a meridian
	(resp. Seifert longitude) of the corresponding component of~$L'$. This homeomorphism~$\phi$ fits in the commutative diagram
	\[
	\begin{tikzcd}[row sep=small]
		\partial\nu(L)\ar[r,"i"] \ar[dd,"\phi" '] & X_L  \ar[d]\\
		& X_C\\
		\partial\nu(L')\ar[r,"i'"] & X_{L'}  \ar[u]
	\end{tikzcd}
	\]
	where the other four maps are inclusions. Our precise claim is that for all~$\omega\in\T^\omega_!$,
	the vertical maps in the induced commutative diagram of vector spaces
	\begin{equation}
		\label{eq:diag}
		\begin{tikzcd}[row sep=small]
			H_1(\partial\nu(L);\C^\omega)\ar[r,"i_1"] \ar[dd,"\phi_*" ',"\simeq"] & H_1(X_L;\C^\omega)  \ar[d,"\simeq"]\\
			& H_1(X_C;\C^\omega)\\
			H_1(\partial\nu(L');\C^\omega)\ar[r,"i_1'"] & H_1(X_{L'};\C^\omega) \ar[u,"\simeq" ']
		\end{tikzcd}
	\end{equation}
	are isomorphisms. This is obvious for~$\phi_*$, but not for the two maps on the right-hand side.
	
	To check this fact, note that the inclusion of the pair~$(S^3\times[0,1],X_L)$ in~$(S^3\times[0,1],X_C)$ induces
	isomorphisms in homology with integer coefficients: this easily follows from excision and the definition of concordance. As a consequence,
	the exact sequence of the triple~$(S^3\times[0,1],X_C,X_L)$ yields~$H_n(X_C,X_L;\Z)=0$ for all~$n$.
	By~\cite[Lemma~2.16]{CNT}, this
	implies that~$H_n(X_C,X_L;\C^\omega)$ vanishes for all~$n$ and all~$\omega\in\T^\omega_!$.
	(Note that this result is stated for~$\omega\in\T^\mu_!\cap(S^1\setminus\{1\})^\mu$,
	but the proof extends verbatim, as it never uses the fact that no coordinate of~$\omega$ is equal to~$1$.)
	By the exact sequence
	of the pair~$(X_C,X_L)$, the inclusion induced map~$H_1(X_L;\C^\omega)\to H_1(X_C;\C^\omega)$ is an isomorphism.
	Since this holds for~$L'$ as well, the vertical maps in~\eqref{eq:diag} are indeed isomorphisms. By the discussion above,
	this implies the equality~$\eta_L(\omega)=\eta_{L'}(\omega)$ for all~$\omega\in\T^\omega_!$.
	
	\medskip
	
	We now turn to the signature. Given two concordant links~$L$ and~$L'$, consider the exterior~$X_C$ of a concordance~$C$ as above. Fix the orientation of~$X_C$ so that the induced orientation on its boundary yields~$X_L\sqcup -X_{L'}\subset\partial X_C$. Observe that~$C$ defines a correspondance between components of~$L$
	and of~$L'$ which preserves colors (by definition) and linking numbers. Since the manifold~$P_L$ only depends on this
	data, we have the identity~$P_L=P_{L'}$.
	This also implies that we have a homeomorphism~$\partial\nu(C)\stackrel{h}{\simeq} C\times S^1$ which restricts to homeomorphisms~$\partial\nu(L)\simeq L\times S^1$ on~$S^3\times\{0\}$ and~$\partial\nu(L')\simeq L'\times S^1$
	on~$S^3\times\{1\}$ defining the standard meridians and Seifert longitudes of~$L$ and~$L'$ (recall Section~\ref{sub:ML}). Let~$W$ be the oriented $4$-manifold obtained by gluing~$X_C$ and~$P_L\times[0,1]$ along~$\partial\nu(C)\subset \partial X_C$ and~$\partial P_L\times[0,1]\subset\partial(P_L\times[0,1])$ via the homeomorphism~$\partial\nu(C)\to\partial P_L\times[0,1]$ given by the composition
	\begin{align*}
		\partial\nu(C)\stackrel{h}{\simeq}  C\times S^1= \bigsqcup_K S^1\times [0,1]\times S^1&\longrightarrow \bigsqcup_K\partial D_K\times S^1\times[0,1]=\partial P_L\times[0,1]\\
		(x,t,y)&\longmapsto(x,y,t)\,.
	\end{align*}
	By construction, this homeomorphism restricts on~$S^3\times\{0,1\}$ to the maps used to construct~$M_L=X_L\cup_\partial -P_L$ and~$M_{L'}=X_{L'}\cup_\partial -P_{L'}$. Therefore, the oriented~4-manifold~$W$ has oriented
	boundary $\partial W=M_L\sqcup -M_{L'}$.
	Moreover, and as mentioned above, the inclusions~$X_L,X_{L'}\subset X_C$ induce isomorphisms~$H_1(X_L;\Z)\simeq H_1(X_C;\Z)\simeq H_1(X_{L'};\Z)$. Hence, the Mayer-Vietoris argument in the proof of~\cite[Lemma~2.11]{CMP23} shows that the homomorphism~$\pi_1(X_C)\to\Z^\mu$ defined by the colors yields a homomorphism~$\Phi\colon\pi_1(W)\to\Z^\mu$ which simultaneously extends meridional homomorphisms~$\varphi\colon\pi_1(M_L)\to\Z^\mu$ and~$\varphi'\colon\pi_1(M_{L'})\to\Z^\mu$. In other words, the equality~$\partial W=M_L\sqcup -M_{L'}$ holds over~$\Z^\mu$.
	
	We now show that the theorem follows from checking that~$\sigma_\omega(W)=0$ for all~$\omega\in\T^\mu_!$.
	This can be done using the definition of~$\sigma_L$ via the auxiliary link~$L^\#$ and $4$-manifold~$W_F$ over~$\Z^\mu$ as in Section~\ref{sub:ext}, but we will use Remark~\ref{rem:rho} instead, which connects the extended signature with the~$\rho$-invariant. Indeed, for all~$\omega\in\T^\mu$, we have the equalities
	\begin{align*}
	\sigma(W)-\sigma_\omega(W)&\stackrel{\eqref{eq:rho-W}}{=}\rho(\partial W,\chi_\omega\circ\Phi)=\rho(M_{L},\chi_\omega\circ\phi)-\rho(M_{L'},\chi_\omega\circ\phi')\\
	&\stackrel{\eqref{eq:rho}}{=}(\sigma(W_F)-\sigma_{L}(\omega))-(\sigma(W_{F'})-\sigma_{L'}(\omega))=\sigma_{L'}(\omega)-\sigma_{L}(\omega)\,.
	\end{align*}
In the last equality, we used the fact that~$\sigma(W_F)$ is determined by the linking numbers and colors of the components of~$L$ (recall Remark~\ref{rem:rho}); since this data coincides for~$L$ and~$L'$, the equality~$\sigma(W_F)=\sigma(W_{F'})$ follows.
 As the element~$\omega=(1,\dots,1)$ belongs to the set~$\T^\mu_!$, it only remains to check that~$\sigma_\omega(W)=0$ for all~$\omega\in\T^\mu_!$.
	
	To do so, let us apply the Novikov-Wall theorem to the decomposition~$W=X_C\cup (P_L\times [0,1])$.
	First note that~$\sigma_\omega(P_L\times[0,1])=0$ for all~$\omega\in\T^\mu$, as the intersection form vanishes identically: by Remark~\ref{rems:int}\ref{rems:int-3}, this follows from the fact that the inclusion~$\partial(P_L\times [0,1])\subset P_L\times [0,1]$ induces epimomorphisms in homology (with~$\C^\omega$-coefficients).
	Recall also that since~$H_n(X_C,X_L;\Z)=0$ for all~$n$, the aforementioned~\cite[Lemma~2.16]{CNT} implies that~$H_n(X_C,X_L;\C^\omega)$ vanishes for all~$n$ and all~$\omega\in\T^\omega_!$; as a consequence, the inclusion
	induced map~$H_2(X_L;\C^\omega)\to H_2(X_C;\C^\omega)$ is an isomorphism, and~$H_2(\partial X_C;\C^\omega)\to H_2(X_C;\C^\omega)$ an epimorphism. By Remark~\ref{rems:int}\ref{rems:int-3}, we have~$\sigma_\omega(X_C)=0$ for all~$\omega\in\T^\omega_!$. Hence, the Novikov-Wall theorem~\eqref{eq:NW} yields
	\[
	\sigma_\omega(W)=\sigma_\omega(X_C)+\sigma_\omega(P_L\times [0,1])+\mathit{Maslov}(\mathcal{L}_-,\mathcal{L}_0,\mathcal{L}_+)=\mathit{Maslov}(\mathcal{L}_-,\mathcal{L}_0,\mathcal{L}_+)\,,
	\]
	and it remains to check that this Maslov index vanishes.
	
	Writing~$(H,\,\cdot\,)$ (resp.~$(H',\,\cdot'\,)$) for the complex vector space~$H_1(\partial\nu(L);\C^\omega)$ (resp.~$H_1(\partial\nu(L');\C^\omega)$) endowed with the
	twisted intersection form, we have Lagrangian subspaces given by the kernels
	of the inclusion induced maps
	\begin{align*}
		V_-&:=\ker\left(H\to H_1(X_{L};\C^\omega)\right)\,,\quad V_-':=\ker\left(H'\to H_1(X_{L'};\C^\omega)\right)\,,\\
		V_+&:=\ker\left(H\to H_1(P_{L};\C^\omega)\right)\,,\quad V_+':=\ker\left(H'\to H_1(P_{L'};\C^\omega)\right)\,.
	\end{align*}
	Since the oriented boundary of~$W$ is~$M_L\sqcup -M_{L'}$, the corresponding intersection form on~$H\oplus H'$
	is given by~$(x,x')\bullet (y,y')=x\cdot y-x'\cdot'y'$.
	With these notations, the Lagrangian subspaces of~$(H\oplus H',\bullet)$ are given by
	\begin{align*}
		\mathcal{L}_-&=\ker\left(H\oplus H'\longrightarrow H_1(X_L;\C^\omega)\oplus H_1(X_{L'};\C^\omega)\right)=V_-\oplus V_-'\\
		\mathcal{L}_0&=\ker\left(H\oplus H'\longrightarrow H_1(\partial\nu(C);\C^\omega)\right)\\
		\mathcal{L}_+&=\ker\left(H\oplus H'\longrightarrow H_1(P_L;\C^\omega)\oplus H_1(P_{L'};\C^\omega)\right)=V_+\oplus V_+'\,.
	\end{align*}
	Since the isomorphism~$\phi_*\colon H\to H'$ from~\eqref{eq:diag} is induced by an orientation-preserving homeomorphism~$\phi\colon\partial\nu(L)\to\partial\nu(L')$, it is an isometry.
	Moreover, it maps the Lagrangians~$V_-$
	onto $V_-'$ and~$V_+$
	onto~$V_+'$, yielding~$\mathcal{L}_-=V_-\oplus\phi_*(V_-)$ and~$\mathcal{L}_+=V_+\oplus\phi_*(V_+)$. Finally, the vector space~$H$ can be explicitly computed as
	\[
	H=\bigoplus_{K\in\mathcal{K}(\omega)}\C\left<m_K,\ell_K\right>\,,
	\]
	where
	\[
	\mathcal{K}(\omega)=\bigcup_{i=1}^\mu\mathcal{K}_i(\omega)\quad\text{and}\quad
	\mathcal{K}_i(\omega)=\{K\subset L_i\mid \textstyle{\omega_i=\prod_{j\neq i}\omega_j^{\lk(K,L_j)}=1}\}\,,
	\]
	and similarly for~$H'$ with the corresponding index set~$\mathcal{K}'(\omega)$. For any~$K\in\mathcal{K}(\omega)$,
	the corresponding~$K'\subset L'$ belongs to~$\mathcal{K}'(\omega)$ and the elements~$m_K-m_{K'}$ and~$\ell_K-\ell_{K'}$ belong to~$\mathcal{L}_0$. By a dimension count, they freely generate this Lagrangian, which is therefore
	equal to
	\[
	\mathcal{L}_0=\{x\oplus-\phi_*(x)\mid x\in H\}\,.
	\]
	We can now apply Lemma~\ref{lem:alg} and get~$\mathit{Maslov}(\mathcal{L}_-,\mathcal{L}_0,\mathcal{L}_+)=0$.
\end{proof}

  \section{Links not concordant to their mirror image}
  \label{sec:ex}
  
  We now present an infinite family of links that are not concordant to their mirror image, a fact that is detected by
  the extended signature, but that cannot be proved
  using any of the following concordance invariants:
  \begin{enumerate}
  \item Non-extended multivariable signatures~\cite{CF08,CNT}.
  \item Multivariable Alexander polynomials~\cite{Kaw78}.
  \item Blanchfield forms over the localised ring~$\Lambda_S$~\cite[Chapter~IX, Theorem~6]{Hil81}.
  \item The linking numbers and Milnor triple linking numbers~\cite{Mil54,Sta65,Cas75}.
\end{enumerate}

However, we shall see that this fact can be detected by the Milnor number~$\mu(1123)$.

\medskip

For any~$n\in\mathbb{N}$, we consider the $3$-component link~$L(n) = K_1\cup K_2\cup K_3$ illustrated in Figure~\ref{fig:brunnian_links}. Note that this is the unlink if~$n=0$, and it is a Brunnian link for all~$n\in\mathbb{N}$. In particular, all the linking numbers are equal to~$0$. To compute most of the aforementioned invariants of~$L(n)$, we will use C-complexes, following~\cite{Cim,CF08,CFT}. Let us refer to~\cite{CF08} for the definition of C-complexes and the associated generalized Seifert matrices, and simply point out here that~$L(n)$ has an evident C-complex~$F(n)$ consisting of three discs and four clasps. This is shown in Figure~\ref{fig:C-complex} in the case~$n=1$, together with two curves representing a basis of~$H_1(F(n);\mathbb{Z}) \cong \mathbb{Z}^2$.

\begin{figure}
	\centering
	\begin{overpic}[width=4cm]{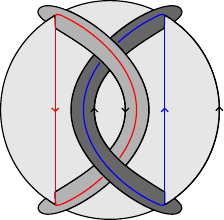}
		\put (15,45) {\color{red}$\alpha$}
		\put (80,45) {\color{blue}$\beta$}
	\end{overpic}
	\caption{The C-complex $F(1)$ and a basis of its first homology.}
	\label{fig:C-complex}
\end{figure}

It is now straightforward to compute that, for any choice of signs~$\varepsilon\in\{\pm 1\}^3$, the associated generalized Seifert matrix is~$A^\varepsilon = \begin{pmatrix} 0 & n \\ n & 0 \end{pmatrix}$. Following~\cite[Section~2]{CF08}, for any~$\omega\in (S^1\setminus\{1\})^3$, we get a matrix $$H(\omega)\coloneqq\sum_{\varepsilon\in\{\pm 1\}^3}\bigg(\prod_{i=1}^3(1-\overline{\omega}_i^{\varepsilon_i})\bigg)A^\varepsilon = \bigg(\prod_{i=1}^{3}(1-\omega_i)(1-\bom_i)~\bigg) \begin{pmatrix} 0 & n \\ n & 0 \end{pmatrix}\,,$$
and the non-extended multivariable signature is thus
$$\sigma_{L(n)}(\omega)\coloneqq\sign(H(\om)) = \pm \sign \begin{pmatrix} 0 & n \\ n & 0 \end{pmatrix} = 0\,.$$
We therefore see that the non-extended multivariable signature does not allow us to distinguish~$L(n)$ from its mirror image~$-L(n)$ (nor~$L(n)$ from~$L(m)$ for~$n\neq m$).

As for the multivariable Alexander polynomial, it can easily be computed (in the normalized version known as Conway potential function) from the generalized Seifert matrices using the main theorem from~\cite{Cim}, and we obtain
\[
\nabla_{L(n)} (t_1,t_2,t_3) = -n^2(t_1-t_1^{-1})^3(t_2-t_2^{-1})(t_3-t_3^{-1})\,.
\]
We therefore see that, for~$n\neq m$,~$L(n)$ and~$L(m)$ are not equivalent, and in fact not concordant~\cite{Kaw78}. However, since for $3$-component links the Conway function is invariant under mirror image, we still cannot distinguish $L(n)$ from its mirror.

Since~$\nabla_{L(n)}\neq 0$ for~$n\neq 0$, in order to compute the Blanchfield form of $L(n)$ over~$\Lambda_S$ one can apply~\cite[Theorem 1.2]{CFT}. Notice, however, that in this case we need to work with a so-called \emph{totally connected} C-complex, a condition which is not satisfyied by the C-complex~$F(n)$ introduced previously. One can easily change~$F(n)$ to a totally connected C-complex by introducing two additional clasps. The computation of the associated matrices being elementary but rather tedious, we will not carry it out here; let us just observe that, in the end, one obtains that the Blanchfield form of~$L(n)$ over~$\Lambda_S$ is represented by a metabolic matrix, and can therefore not distinguish~$L(n)$ from~$-L(n)$.

Finally, as we have already mentioned, all the linking numbers of $L(n)$ vanish, while Milnor's triple linking number~$\mu(123)$ is invariant under mirror image for links with $3$ components~\cite{Mil57}. In fact, it is not hard to show that for~$L(n)$ we have~$\mu(123)=0$ (one can for instance apply the results of~\cite{MM03} to the C-complex $F(n)$).

\medskip

We now turn to the computation of the extended signature. By~\cite[Corollary 4.8]{CMP23}, since all the linking numbers vanish and~$L(n)'\coloneq K_2\cup K_3$ is the unlink, for any~$\omega\in(S^1\setminus\{1\})^2$ we have $$\sigma_{L(n)}(1,\omega) = \sigma_{L(n)'}(\omega) + \sgn ((K_1/L(n)')(\omega)) = \sgn ((K_1/L(n)')(\omega))\,,$$ where~$(K_1/L(n)')(\omega)\in\mathbb{R}\cup\{\infty\}$ is the so-called \textit{slope} of~$K_1\cup L(n)'$, an invariant of links with a distinguished component introduced in~\cite{DFL}. The computation of the extended signature thus reduces to the computation of the slope, which can be performed using C-complexes, following~\cite{DFL3}. For that purpose, we need to find a C-complex for~$L(n)'$ disjoint from~$K_1$; such a C-complex~$F(n)'$, consisting of two disjoint tori, is shown in Figure~\ref{fig:surface_slope} for~$n=1$, together with a basis of its first homology. 
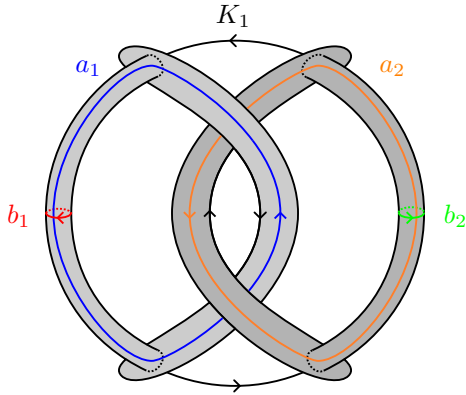
\begin{figure}
	\centering
	\begin{overpic}[width=5cm]{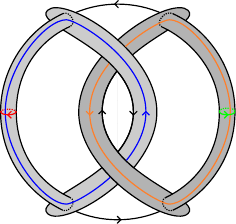}
		\put (-10,45) {\color{red}$b_1$}
		\put (8,85) {\color{blue}$a_1$}
		\put (88,85) {\color{orange}$a_2$}
		\put (105,45) {\color{green}$b_2$}
		\put (45,98) {$K_1$}
	\end{overpic}
	\caption{The C-complex $F(1)'$ and a basis of its first homology.}
	\label{fig:surface_slope}
\end{figure}

The associated Seifert matrices can be easily computed to be
$$ A^{++} = \begin{pmatrix}
	0&0&n&0\\1&0&0&0\\n&0&0&0\\0&0&1&0\end{pmatrix}\,,\, A^{+-} = \begin{pmatrix}
	0&0&n&0\\1&0&0&0\\n&0&0&1\\0&0&0&0\end{pmatrix}\,,\, A^{-+} = (A^{+-})^T\, \text{ and }\,A^{--}=(A^{++})^T,$$
so, for~$\omega = (\omega_1,\omega_2)\in(S^1\setminus\{1\})^2$, using the identity~$(1-\om_i)^{-1}+(1-\bom_i)^{-1} = 1$ we obtain 
$$ E(\omega)\coloneq \sum_{\varepsilon\in\{\pm 1\}^2}\bigg(\prod_{i=1}^2(1-\omega_i^{\varepsilon_i})^{-1}\bigg)A^\varepsilon = \begin{pmatrix}
	0&(1-\bom_1)^{-1}&n&0\\(1-\om_1)^{-1}&0&0&0\\n&0&0&(1-\bom_2)^{-1}\\0&0&(1-\om_2)^{-1}&0\end{pmatrix}\,. $$
By Alexander duality,~$E(\om)$ can naturally be considered as an operator $$E(\om):H_1(F(n)';\mathbb{C})\rightarrow H_1(S^3\setminus F(n)';\mathbb{C})\cong H^1(F(n)';\mathbb{C})\cong \operatorname{Hom}(H_1(F(n)';\mathbb{C}),\mathbb{C})\,,$$
represented as a matrix in the basis~$\{a_1,b_1,a_2,b_2\}$ of~$H_1(F(n)';\mathbb{C})$ and its dual basis of~$H^1(F(n)';\mathbb{C})$. Similarly,~$[K_1] = b_1^*+b_2^* \in H_1(S^3\setminus F(n)';\mathbb{C})\cong H^1(F(n)';\mathbb{C})$. Setting 
$$\alpha = (1-\om_1)a_1 - n(1-\om_1)(1-\bom_2)b_2 + (1-\om_2)a_2 - n(1-\bom_1)(1-\om_2)b_1 \in H_1(F(n)';\mathbb{C})\,,$$
we obtain~$E(\om)\alpha = b_1^*+b_2^* = [K_1]$, so that~$[K_1]\in \operatorname{Im}(E(\om))$. Since moreover~$\operatorname{det}E(\om) \neq 0$ for~$\omega = (\omega_1,\omega_2)\in(S^1\setminus\{1\})^2$, we see that~$\operatorname{Ker}(E(\om)) = 0$. By~\cite[Theorem 4.3]{DFL3}, we thus obtain $$(K_1/L(n)')(\om) = -K_1(\alpha) = n(1-\om_1)(1-\bom_2)+n(1-\bom_1)(1-\om_2) = 2n \Re((1-\om_1)(1-\bom_2))\,.$$
Therefore,
$$\sigma_{L(n)}(1,\om) = \sgn((K_1/L(n)')(\om)) = \sgn[\Re((1-\om_1)(1-\bom_2))]\,,$$ which is non-zero for infinitely many~$\om\in\T^2_!$. By Theorem~\ref{thm:conc} and Proposition~\ref{prop:sign-L}, we can finally conclude that~$L(n)$ is not concordant to~$-L(n)$.

\begin{rems}\label{rems:example}
	\begin{enumerate}
		\item Of course, our computations show that the slope as well can distinguish~$L(n)$ from~$-L(n)$, being itself a concordance invariant~\cite{DFL3}. We thus have infinitely many pairs of links that can be distinguished by the slope but not by the Conway functions of any of their sublinks. While the existence of such examples is not surprising, to the best of our knowledge this is the first explicit family appearing in the literature.
		
		\item By Cochran's theory of derived invariants~\cite{Coc90}, the surface~$F(n)'$ can also be used to compute the Milnor number~$\mu(1123)$ of $L(n)$. Applying~\cite[Proposition 6.5]{Coc90}, one easily gets that~$L(n)$ has~$\mu(1123) = \pm n$. Since~$\mu(1123)$ changes sign after mirror image and is a concordant invariant, we see that this Milnor number is enough to distinguish~$L(n)$ from~$-L(n)$.
	\end{enumerate}
\end{rems}
  
\bibliographystyle{plain}
\bibliography{bibliography}
	
\end{document}